\documentclass[11pt]{amsart}
\usepackage{amsmath}
\usepackage{amsfonts}
\usepackage{amsthm}
\usepackage{amssymb}
\usepackage{tikz-cd}
\usepackage{graphics}
\usepackage{array}
\usepackage{enumerate}
\usepackage{dsfont}
\usepackage{color}
\usepackage{wasysym}
\usepackage{hyperref}
\usepackage{pdfsync}

\newcommand{\bel}[1]{\begin{equation}\label{#1}}

\newcommand{\be}{\begin{equation}}

\newcommand{\ba}{\begin{eqnarray}}
\newcommand{\ea}{\end{eqnarray}}

\newcommand{\qe}{\end{equation}}

\newcommand{\N}{{\mathbb N}}

\newcommand{\C}{{\mathbb C}}
\newcommand{\Q}{{\mathbb Q}}

\DeclareMathOperator{\Alb}{Alb}
\DeclareMathOperator{\Image}{Im}
\DeclareMathOperator{\Pic}{Pic}
\DeclareMathOperator{\Ker}{Ker}
\DeclareMathOperator{\codim}{codim}
\DeclareMathOperator{\Supp}{Supp}

\DeclareMathOperator{\GV}{GV}
\DeclareMathOperator{\IT}{IT_0}
\DeclareMathOperator{\Exc}{Exc}
\DeclareMathOperator{\id}{id}

\newcommand{\Hmm}[1]{\leavevmode{\marginpar{\tiny%
$\hbox to 0mm{\hspace*{-0.5mm}$\leftarrow$\hss}%
\vcenter{\vrule depth 0.1mm height 0.1mm width \the\marginparwidth}%
\hbox to
0mm{\hss$\rightarrow$\hspace*{-0.5mm}}$\\\relax\raggedright #1}}}

\theoremstyle{plain}

\newtheorem{thm}{Theorem}[section]

\newtheorem{prop}[thm]{Proposition}

\newtheorem{coro}[thm]{Corollary}
\newtheorem{lemma}[thm]{Lemma}

\theoremstyle{definition}

\newtheorem{defi}[thm]{Definition}
\newtheorem{rem}[thm]{Remark}
\newtheorem*{ac}{Acknowledgements}

\begin{document}

\title[Pushforwards of klt pairs]{Pushforwards of klt pairs under morphisms to abelian varieties}
	
	\author{Fanjun Meng}
	\address{Department of Mathematics, Northwestern University, 2033 Sheridan Road, Evanston, IL 60208, USA}
	\email{fanjunmeng2022@u.northwestern.edu}


\begin{abstract}
Let $f$ be a morphism from a klt pair $(X, \Delta)$ to an abelian variety $A$, $m\geq1$ a rational number and $D$ a Cartier divisor on $X$ such that $D\sim_{\Q}m(K_X+\Delta)$. We prove that the sheaf $f_*\mathcal{O}_X(D)$ becomes globally generated after pullback by an isogeny and has the Chen-Jiang decomposition, along with some related results. These are applied to some effective results for $\mathcal{O}_X(D)$ when $X$ is irregular.
\end{abstract}

\maketitle
	\setcounter{tocdepth}{1}
	\tableofcontents

\section{Introduction}

In this paper, we prove several results about pushforwards of klt pairs under morphisms to abelian varieties over $\C$ and give some applications. These results are natural generalizations of \cite{LPS20} in the presence of singularities.

Many facts are known about the positivity properties of pushforwards of pluricanonical bundles under morphisms from smooth projective varieties to abelian varieties. For example, they are GV-sheaves by \cite{GL87, Hac04, PP11a, PS14}. Their cohomological support loci are finite unions of torsion subvarieties by \cite{GL91, Sim93, Lai11, LPS20}. They have the Chen-Jiang decomposition by \cite{CJ18, PPS17, LPS20}. For the definitions of the concepts mentioned above, we refer to Section \ref{2}. It is natural to ask what happens if we allow singularities. Our work treats the case of klt pairs.

The first three theorems stated below are the main results of the paper. They are all equivalent by Proposition \ref{equi}. For the first result, the case of pushforwards of canonical bundles and pluricanonical bundles is obtained in \cite{CJ18, PPS17, LPS20} in increasing generality.

\begin{thm}\label{main1}
Let $f$ be a morphism from a klt pair $(X, \Delta)$ to an abelian variety $A$, $m\geq1$ a rational number and $D$ a Cartier divisor on $X$ such that $D\sim_{\Q}m(K_X+\Delta)$. Then there exists an isogeny $\varphi\colon A'\to A$ such that $\varphi^*f_*\mathcal{O}_X(lD)$ is globally generated for every $l\geq 1$.	
	\begin{center}
	\begin{tikzcd}
			X' \arrow[r, "\varphi'"] \arrow[d, "f'"] & X \arrow[d, "f"] \\
			 A' \arrow[r, "\varphi" ] & A 
	\end{tikzcd}
	\end{center}
\end{thm}

We first need to do substantial work in Section \ref{3} to perform some reduction steps which imply that it suffices to prove Theorem \ref{main1} for the Cartier divisor $ND$ for some positive integer $N$, in order to deduce it for $D$ itself. Then we adapt the strategy in \cite[Sections 8 and 9]{LPS20} to our setting to prove Proposition \ref{SNC}. The proof relies in part on analytic results based on the theory of singular Hermitian metrics from \cite{CP17, HPS18}. See Lemma \ref{split} and Proposition \ref{SNC} for details. The invariance of plurigenera for smooth families of smooth varieties is used in the proof of the main theorems in \cite{LPS20}. Since we do not know it in our case, we use a different technique to avoid it. See Remark \ref{genera} for details.

Theorem \ref{main2} is a consequence of Theorem \ref{main1} and Proposition \ref{equi}. The case of pushforwards of pluricanonical bundles is obtained in \cite[Theorem A]{LPS20}.

\begin{thm}\label{main2}
Let $f$ be a morphism from a klt pair $(X, \Delta)$ to an abelian variety $A$, $m\geq1$ a rational number and $D$ a Cartier divisor on $X$ such that $D\sim_{\Q}m(K_X+\Delta)$. Then there exists a generically finite surjective morphism $h\colon Z\to X$ from a smooth projective variety $Z$ such that $f_*\mathcal{O}_X(D)$ is a direct summand of $(f\circ h)_*\mathcal{O}_Z(K_Z)$.
\end{thm}

Theorem \ref{main2} implies that all the properties of pushforwards of canonical bundles mentioned above, along with many other properties, carry over to pushforwards of klt pairs. See Section \ref{4} for details.

The third theorem regards the Chen-Jiang decomposition property for the sheaf $f_*\mathcal{O}_X(D)$. The case of pushforwards of canonical bundles and pluricanonical bundles is obtained in \cite{CJ18, PPS17, LPS20} in increasing generality.

\begin{thm}\label{main3}
Let $f$ be a morphism from a klt pair $(X, \Delta)$ to an abelian variety $A$, $m\geq1$ a rational number and $D$ a Cartier divisor on $X$ such that $D\sim_{\Q}m(K_X+\Delta)$. Then $f_*\mathcal{O}_X(D)$ admits a finite direct sum decomposition
$$f_*\mathcal{O}_X(D)\cong \bigoplus_{i\in I}(\alpha_i\otimes p_i^*\mathcal{F}_i),$$
where each $A_i$ is an abelian variety, each $p_i\colon A\to A_i$ is a fibration, each $\mathcal{F}_i$ is a nonzero M-regular coherent sheaf on $A_i$, and each $\alpha_i\in\Pic^0(A)$ is line bundle which becomes trivial when pulled back by the isogeny $\varphi$ in Theorem \ref{main1}.
\end{thm}

When attacking these theorems in Section \ref{3}, we in fact first prove Theorem \ref{main3}. This theorem gives a detailed description about the positivity of the sheaf $f_*\mathcal{O}_X(D)$. It essentially says that the sheaf $f_*\mathcal{O}_X(D)$ is not just semipositive, but semiample, since M-regular sheaves are ample by \cite[Corollary 3.2]{Deb06}.

We briefly explain how the main theorems are proved. When we consider pluricanonical bundles on smooth varieties, Viehweg's cyclic covering trick applies in order to perform some reduction steps. However, this trick does not apply in our case when $m>1$. Instead, in Section \ref{3} we need to do substantial work in order to overcome this, using a technique from \cite[Theorem 1.7]{PS14} and some other methods. These statements imply that it suffices to show that $f_*\mathcal{O}_X(ND)$ satisfies the conclusions of Theorems \ref{main1}, \ref{main2} and \ref{main3} for some positive integer $N$ in order to prove that $f_*\mathcal{O}_X(D)$ has the same properties, which gives us extra flexibility. When $N$ is sufficiently big and divisible, we follow the strategy in \cite[Sections 8 and 9]{LPS20}, with some modifications, to prove that $f_*\mathcal{O}_X(ND)$ has the required properties. See Proposition \ref{SNC} for details. 

We note that there is an alternative approach to proving the main theorems for $f_*\mathcal{O}_X(ND)$ when $N$ is sufficiently big and divisible, different from the one described above. We believe that this is worth mentioning since the techniques are rather different and interesting in their own right. It is based on the use of the minimal model program but it requires the assumption that the general fiber $(F, \Delta|_F)$ of $f$ has a good minimal model. This approach is not purely algebraic either. See Proposition \ref{MMP} and Remark \ref{not} for details.

With more work, in Section \ref{4} we prove that pushforwards of klt pairs under the Albanese morphism satisfy stronger positivity which is valid only for $m>1$. The crucial point is the next theorem which generalizes \cite[Proposition 2.12]{HP02}, \cite[Lemma 2.2]{Jia11} and \cite[Theorem 11.2]{HPS18} to klt pairs. We consider a smooth model of the Iitaka fibration associated to a Cartier divisor $D$ where $\kappa(X, D)\geq0$. 

\begin{thm}\label{main6}
Let $f\colon X\to Y$ be a smooth model of the Iitaka fibration associated to a Cartier divisor $D$ on $X$ where $D\sim_{\Q}m(K_X+\Delta)$, $m>1$ is a rational number, $Y$ is smooth and $(X, \Delta)$ is a klt pair. Let $a_Y\colon Y\to \Alb(Y)$ be the Albanese morphism of $Y$ and $g$ the morphism $a_Y\circ f$. Then:
\begin{enumerate}
	\item[$\mathrm{(i)}$] For every torsion point $\alpha\in\Pic^0(X)$, every $\beta\in\Pic^0(Y)$ and every nef divisor $L$ on $\Alb(Y)$, we have
$$\quad h^0(X, \mathcal{O}_X(D+g^*L)\otimes\alpha)=h^0(X, \mathcal{O}_X(D+g^*L)\otimes\alpha\otimes f^*\beta).$$

	\item[$\mathrm{(ii)}$] There exist finitely many torsion points $\alpha_i\in\Pic^0(X)$ such that
$$V^0(X, \mathcal{O}_X(D))=\bigcup_{i\in I}(\alpha_i\otimes f^*\Pic^0(Y)).$$
\end{enumerate}
\end{thm}

To prove Theorem \ref{main6}, we follow the strategy used for example in \cite[Lemma 2.2]{Jia11} and \cite[Theorem 11.2]{HPS18} for the case of smooth varieties. However, the techniques involved are somewhat different in our case. Once we have Theorem \ref{main6}, the next theorem below follows from Theorem \ref{main3} and standard arguments. The case when $D=mK_X$ and $m\geq2$ is an integer is obtained in \cite[Theorem D]{LPS20}.

\begin{thm}\label{main4}
Let $f\colon X\to Y$ be a smooth model of the Iitaka fibration associated to a Cartier divisor $D$ on $X$ where $D\sim_{\Q}m(K_X+\Delta)$, $m>1$ is a rational number, $Y$ is smooth and $(X, \Delta)$ is a klt pair. Let $a_X\colon X\to \Alb(X)$ be the Albanese morphism of $X$ and $a_f\colon \Alb(X)\to\Alb(Y)$ the induced morphism between Albanese varieties. Then $(a_X)_*\mathcal{O}_X(lD)$ admits, for every positive integer $l$, a finite direct sum decomposition 
$$(a_X)_*\mathcal{O}_X(lD)\cong \bigoplus_{i\in I}(\alpha_i\otimes a_f^*\mathcal{F}_i),$$
where each $\mathcal{F}_i$ is a coherent sheaf on $\Alb(Y)$ satisfying $\IT$ and each $\alpha_i\in\Pic^0(X)$ is a torsion line bundle whose order can be bounded independently of $l$.
\end{thm}

\begin{rem}
After finishing this manuscript, I was informed by Zhi Jiang that he has also proved a result in \cite{Jia20} which is essentially the same as Theorem \ref{main4}, by a different but related method. He uses it to obtain some very interesting geometric applications. I would like to thank him for sharing his draft.
\end{rem}

We apply Theorem \ref{main4} to studying effective freeness and very ampleness for klt pairs on irregular varieties in Section \ref{4}. The following statement is a special case of Theorem \ref{main5}. It deals with the case when $(X, \Delta)$ is of log general type and $X$ is of maximal Albanese dimension, and generalizes some effective results about varieties with at worst canonical singularities from \cite{PP03, PP11b, LPS20} to klt pairs. See \cite[Corollary E]{LPS20} for instance.

\begin{coro}\label{coro5}
Let $(X, \Delta)$ be a klt pair of log general type, $D$ a Cartier divisor on $X$ such that $D\sim_{\Q}m(K_X+\Delta)$, $m>1$ a rational number and $a_X\colon X\to \Alb(X)$ the Albanese morphism of $X$. Assume that $a_X$ is generically finite onto its image and denote the union of the positive dimensional fibers of $a_X$ by $\Exc(a_X)$. Then for every $\alpha\in\Pic^0(X)$:
\begin{enumerate}
	\item[$\mathrm{(i)}$] $\mathcal{O}_{X}(2D)\otimes\alpha$ is globally generated away from $\Exc(a_X)$.

	\item[$\mathrm{(ii)}$] $\mathcal{O}_{X}(3D)\otimes\alpha$ is very ample away from $\Exc(a_X)$.
\end{enumerate}
\end{coro}

Finally, it is natural to ask what happens if we consider log canonical pairs instead of klt pairs. Although some of the ideas we use here still work, at this point we do not know how to fully generalize our theorems to this case even when $m=1$. Moreover, we do not have the generalization of Proposition \ref{SNC} to log canonical pairs, because the results from the singular Hermitian metrics are not yet known in this setting.

\begin{ac}
{I would like to express my sincere gratitude to my advisor Mihnea Popa for proposing this problem and for helpful discussions and generous support. I would also like to thank Bingyi Chen for helpful discussions, and Zhi Jiang for sharing \cite{Jia20}.}
\end{ac}

\section{Preliminaries}\label{2}
We work over $\C$ and all varieties are projective throughout the paper. A \emph{fibration} is a projective surjective morphism with connected fibers. A \emph{birational contraction} is a birational map whose inverse does not contract any divisor. For the definitions and basic results on the singularities of pairs and the minimal model program (MMP) we refer to \cite{KM98}. We always ask the boundary $\Delta$ in a pair $(X,\Delta)$ to be effective.

First, we give the definition for the general fiber of a morphism which is not necessarily surjective.

\begin{defi}\label{general fiber}
Let $f\colon X \to Y$ be a morphism between two normal projective varieties. The Stein factorization of $f$ gives a decomposition of $f$ as $g\circ h$ where $h$ is a fibration and $g$ is a finite morphism. The \emph{general fiber} of $f$ is defined as the general fiber of $h$. The definition is similar when $X$ is replaced with a log canonical pair $(X, \Delta)$.
\end{defi}	

We include the definition of good minimal models.

\begin{defi} 
Let $(X, \Delta)$ be a log canonical pair over a normal projective variety $Z$. A birational contraction $\xi \colon (X, \Delta) \dashrightarrow (Y, \Delta_Y)$ over $Z$ to a $\Q$-factorial log canonical pair $ (Y,\Delta_Y) $ is a \emph{good minimal model} over $Z$ of the pair $(X,\Delta)$ if $ \Delta_Y =\xi_*\Delta$, if $K_Y+\Delta_Y$ is semiample over $Z$ and if
$$a(F, X, \Delta) < a(F, Y, \Delta_Y)$$
for any prime divisor $F$ on $X$ which is contracted by $\xi$.
\end{defi}

Next we recall the definition for the irregularity of a projective variety $X$.

\begin{defi}
Let $X$ be a smooth projective variety. The \emph{irregularity} $q(X)$ is defined as $h^1(X,\mathcal{O}_X)$. If $X$ is a projective variety, the \emph{irregularity} $q(X)$ is defined as the irregularity of any resolution of $X$.
\end{defi}

If $X$ is smooth, the irregularity $q(X)$ is equal to the dimension of its Albanese variety $\Alb(X)$.

Let $A$ be an abelian variety and $\hat{A}\cong \Pic^0(A)$ its dual abelian variety. We denote by
$$\bold{R}\hat{\mathcal{S}}\colon \bold{D}(A)\to \bold{D}(\hat{A}),\quad \bold{R}\hat{\mathcal{S}}\mathcal{F}:=\bold{R}p_{2*}(p_1^*\mathcal{F}\otimes P)$$
the Fourier-Mukai functor induced by a normalized Poincar\'e bundle $P$ on $A\times\hat{A}$ where $p_1$ and $p_2$ are projections onto $A$ and $\hat{A}$ respectively.

We recall several definitions. 

\begin{defi}
A coherent sheaf $\mathcal{F}$ on an abelian variety $A$
\begin{enumerate}
	\item[$\mathrm{(i)}$] is a GV-\emph{sheaf} if $\codim$ $\Supp$ $\bold{R}^i\hat{\mathcal{S}}\mathcal{F}\geq i$ for every $i>0$.
	\item[$\mathrm{(ii)}$] is \emph{M-regular} if $\codim$ $\Supp$ $\bold{R}^i\hat{\mathcal{S}}\mathcal{F}> i$ for every $i>0$.
	\item[$\mathrm{(iii)}$] \emph{satisfies} $\IT$ if $\bold{R}^i\hat{\mathcal{S}}\mathcal{F}=0$ for every $i>0.$
\end{enumerate}
\end{defi}

We include a proposition about determining when a GV-sheaf satisfies $\IT$ which is \cite[Proposition 2.3]{LPS20}.

\begin{prop}\label{IT}
Let $\mathcal{F}$ be a $\GV$-sheaf on an abelian variety $A$. If $h^0(A, \mathcal{F}\otimes\alpha)$ is independent of $\alpha\in\Pic^0(A)$, then $\mathcal{F}$ satisfies $\IT$ and $\bold{R}\hat{\mathcal{S}}\mathcal{F}$ is locally free.
\end{prop}

\begin{defi}
Let $\mathcal{F}$ be a coherent sheaf on an abelian variety $A$. The \emph{cohomological support loci} $V_l^i(A, \mathcal{F})$ for $i\in\N$ and $l\in\N$ are defined by
$$V_l^i(A, \mathcal{F})=\{\alpha\in\Pic^0(A)\mid\dim H^i(A, \mathcal{F}\otimes\alpha)\geq l\}.$$
We use $V^i(A, \mathcal{F})$ to denote $V_1^i(A, \mathcal{F})$.
\end{defi}

\begin{defi}
Let $A$ be an abelian variety. A \emph{torsion subvariety} of $A$ is a translate of an abelian subvariety of $A$ by a torsion point which is a closed point of finite order in $A$.
\end{defi}

Next we define the useful notion of continuous evaluation morphisms (cf. \cite{PP03, LPS20}) and give a simple lemma about it (cf. \cite[Proposition 3.1]{Deb06}).

\begin{defi}
Let $\mathcal{F}$ be a coherent sheaf on an abelian variety $A$. The \emph{continuous evaluation morphism} associated to the coherent sheaf $\mathcal{F}$ is defined by
$$e_{\mathcal{F}}\colon \bigoplus_{\alpha\in \Pic^0(A)\atop \mathrm{torsion}}H^0(A, \mathcal{F}\otimes\alpha)\otimes\alpha^{-1}\to \mathcal{F}$$
which is induced from the evaluation morphisms.
\end{defi}

\begin{lemma}\label{tor}
Let $\mathcal{F}$ be a coherent sheaf on an abelian variety $A$. Then there exists an isogeny $\varphi\colon A'\to A$ such that $\varphi^*\mathcal{F}$ is globally generated if and only if the continuous evaluation morphism $e_{\mathcal{F}}$ associated to $\mathcal{F}$ is surjective.
\end{lemma}

\begin{proof}
First, we prove the if part. Since $\mathcal{F}$ is coherent, there exist finitely many torsion line bundles $\alpha_1, \dots, \alpha_N$ such that
$$\bigoplus^{N}_{i=1}H^0(A, \mathcal{F}\otimes\alpha_i)\otimes\alpha_i^{-1}\to \mathcal{F}$$
is surjective. Then we can choose an isogeny $\varphi$ on $A$ such that $\varphi^*\alpha_i$ becomes trivial for each $1\leq i\leq N$ and thus $\varphi^*\mathcal{F}$ is globally generated. 

We prove the only if part now. Consider the sheaf $\mathcal{G}:=\Image e_{\mathcal{F}}\subseteq \mathcal{F}$. By the same argument, we can choose an isogeny $\psi$ such that $\psi^*\mathcal{G}$ is globally generated. By taking a fiber product of $\psi$ and $\varphi$, we can get a new isogeny $\xi\colon B\to A$ such that both $\xi^*\mathcal{G}$ and $\xi^*\mathcal{F}$ are globally generated. Since $\xi$ is \'{e}tale, we have that $\xi^*\mathcal{G}\subseteq\xi^*\mathcal{F}$. We only need to prove $\xi^*\mathcal{G}=\xi^*\mathcal{F}$ to deduce $\mathcal{G}=\mathcal{F}$ since $\xi$ is faithfully flat. We know that 
$$\xi_*\mathcal{O}_B\cong\bigoplus_{\alpha\in\Ker \hat{\xi}}\alpha$$
where $\hat{\xi}$ is the dual isogeny of $\xi$. We have that $H^0(A, \mathcal{G}\otimes\alpha)=H^0(A, \mathcal{F}\otimes\alpha)$ for every torsion line bundle $\alpha\in\Pic^0(A)$ by the definition of $\mathcal{G}$. By the base change theorem, we have 
$$H^0(B, \xi^*\mathcal{G})\cong H^0(A, \mathcal{G}\otimes\xi_*\mathcal{O}_B)\cong\bigoplus_{\alpha\in\Ker \hat{\xi}}H^0(A, \mathcal{G}\otimes\alpha)$$
$$\cong\bigoplus_{\alpha\in\Ker \hat{\xi}}H^0(A, \mathcal{F}\otimes\alpha)\cong H^0(A, \mathcal{F}\otimes\xi_*\mathcal{O}_B)\cong H^0(B, \xi^*\mathcal{F}).$$
Thus we have $\xi^*\mathcal{G}=\xi^*\mathcal{F}$ since $\xi^*\mathcal{G}$ and $\xi^*\mathcal{F}$ are globally generated.
\end{proof}

We now give the definition of the Chen-Jiang decomposition which is one of the main properties we will discuss in the following sections. The paper \cite{CJ18} justifies the name of this property.

\begin{defi}
Let $\mathcal{F}$ be a coherent sheaf on an abelian variety $A$. The sheaf $\mathcal{F}$ is said to have the \emph{Chen-Jiang decomposition} if $\mathcal{F}$ admits a finite direct sum decomposition
$$\mathcal{F}\cong \bigoplus_{i\in I}(\alpha_i\otimes p_i^*\mathcal{F}_i),$$
where each $A_i$ is an abelian variety, each $p_i\colon A\to A_i$ is a fibration, each $\mathcal{F}_i$ is a nonzero M-regular coherent sheaf on $A_i$, and each $\alpha_i\in\Pic^0(A)$ is a torsion line bundle.
\end{defi}

We include a useful proposition which is \cite[Proposition 3.6]{LPS20}.

\begin{prop}\label{sum}
Let $\mathcal{F}\cong \mathcal{F'}\oplus\mathcal{F''}$ be a coherent sheaf on an abelian variety $A$. Then $\mathcal{F}$ has the Chen-Jiang decomposition if and only if both $\mathcal{F'}$ and $\mathcal{F''}$ have the Chen-Jiang decomposition.
\end{prop}

Next, we give a lemma about the splitting of morphisms which is based on the minimal extension property of some special singular Hermitian metrics. For clarity, we will use the analytic language for the statement and the proof of it. For details, we refer to \cite{CP17, HPS18}.

\begin{lemma}\label{split}
Let $f$ be a projective surjective morphism from a log smooth klt pair $(X, \Delta)$ on a projective manifold $X$ to an abelian variety $A$. Then for every positive integer $N$ which is sufficiently big and divisible such that $f_*\mathcal{O}_X(N(K_X+\Delta))\neq0$, the sheaf $f_*\mathcal{O}_X(N(K_X+\Delta))$ admits a singular Hermitian metric with semipositive curvature and the minimal extension property. In particular, any nonzero morphism $f_*\mathcal{O}_X(N(K_X+\Delta))\to \mathcal{O}_A$ splits.
\end{lemma}

\begin{proof}
We choose $N$ sufficiently divisible such that $N\Delta$ is an integral divisor. Define a new divisor $L_N:=(N-1)K_{X/A}+N\Delta$. By the discussion at the beginning of \cite[Section 4]{CP17}, if $N$ is sufficiently big, then there exists a singular Hermitian metric $h_N$ on the line bundle $\mathcal{O}_X(L_N)$ with semipositive curvature such that the inclusion
$$f_*(\mathcal{O}_X(K_{X/A}+L_N)\otimes \mathcal{I}(h_N))\hookrightarrow f_*\mathcal{O}_X(K_{X/A}+L_N)$$
is generically an isomorphism where $\mathcal{I}(h_N)\subseteq \mathcal{O}_X$ is the multiplier ideal sheaf associated to the singular Hermitian metric $h_N$. Thus we know the sheaf $f_*\mathcal{O}_X(N(K_X+\Delta))\cong f_*\mathcal{O}_X(K_{X/A}+L_N)$ admits a singular Hermitian metric with semipositive curvature and the minimal extension property by \cite[Theorem 21.1 and Corollary 21.2]{HPS18}. Then the property of splitting follows from \cite[Theorem 26.4]{HPS18}.
\end{proof}

Let $\varphi\colon A'\to A$ be an isogeny in Lemma \ref{split}. We have the following base change diagram.
	\begin{center}
	\begin{tikzcd}
			X' \arrow[r, "\varphi'"] \arrow[d, "f'"] & X \arrow[d, "f"] \\
			 A' \arrow[r, "\varphi" ] & A 
	\end{tikzcd}
	\end{center}
We consider the log smooth klt pair $(X', \Delta')$ given by $K_{X'}+\Delta'=\varphi'^*(K_X+\Delta)$. Then the same $N$ as in $f_*\mathcal{O}_X(N(K_X+\Delta))$ will make Lemma \ref{split} work for $f'_*\mathcal{O}_{X'}(N(K_{X'}+\Delta'))$ by the construction of the singular Hermitian metric on it. See \cite[Section 4]{CP17} for details.

\section{Positivity of pushforwards of klt pairs}\label{3}

In this section, we prove our main theorems. The next four statements are needed for the reduction steps. We start with a basic lemma which treats the case when $m=1$.

\begin{lemma}\label{m=1}
Let $f$ be a morphism from a klt pair $(X, \Delta)$ to an abelian variety $A$ and $D$ a Cartier divisor on $X$ such that $D\sim_{\Q}K_X+\Delta$. Then there exists a generically finite surjective morphism $h\colon Z\to X$ from a smooth projective variety $Z$ such that $R^if_*\mathcal{O}_X(D)$ is a direct summand of $R^i(f\circ h)_*\mathcal{O}_Z(K_Z)$ for every $i\in \N$. In particular, the sheaf $R^if_*\mathcal{O}_X(D)$ has the Chen-Jiang decomposition for every $i\in \N$.
\end{lemma}

\begin{proof}
We take a log resolution of $(X, \Delta)$ denoted by $\mu\colon Y\to X$ as in the following diagram. 
	\begin{center}
	\begin{tikzcd}
			Y \arrow[r, "\mu"] \arrow[dr, "g" swap] & X \arrow[d, "f"] \\
			& A 
	\end{tikzcd}
	\end{center}
Then we have
$$K_Y+\Delta_Y\sim_\mathbb{Q}\mu^*(K_{X}+\Delta)+E,$$
where the $\Q$-divisors $\Delta_Y$ and $E$ are effective and have no common components and $E$ is $\mu$-exceptional. We also have
$$K_Y+\Delta_Y+\lceil E\rceil-E\sim_\mathbb{Q}\mu^*(K_{X}+\Delta)+\lceil E\rceil.$$
Let $\Delta'_Y$ be $\Delta_Y+\lceil E\rceil-E$, then the pair $(Y, \Delta'_Y)$ is klt and log smooth since $(X, \Delta)$ is klt. We know that $\mu_*\mathcal{O}_Y(\mu^*D+\lceil E\rceil)\cong \mathcal{O}_X(D)$ since $\lceil E\rceil$ is $\mu$-exceptional and effective. By \cite[Corollary 10.15]{Kol95}, $R^i\mu_*\mathcal{O}_Y(\mu^*D+\lceil E\rceil)$ is torsion-free for every $i\in \N$ since $\mu^*D+\lceil E\rceil\sim_{\Q}K_Y+\Delta'_Y$ and the pair $(Y, \Delta'_Y)$ is klt and log smooth. Thus $R^i\mu_*\mathcal{O}_Y(\mu^*D+\lceil E\rceil)=0$ for $i>0$, since it is $0$ on an open subset of $X$ and torsion-free. Then we deduce that
$$R^ig_*\mathcal{O}_Y(\mu^*D+\lceil E\rceil)\cong R^if_*(\mu_*\mathcal{O}_Y(\mu^*D+\lceil E\rceil))\cong R^if_*\mathcal{O}_X(D)$$
for every $i\in \N$ by Grothendieck spectral sequence. Thus we can assume $(X, \Delta)$ is log smooth from the start.

We can choose a positive integer $N$ such that $ND\sim N(K_X+\Delta)$ and $N\Delta$ becomes an integral divisor. Then we have  $N(D-K_X)\sim N\Delta$. Thus we can take the associated branched covering along $N\Delta$ and resolve the singularities. This gives us a generically finite surjective morphism $h\colon Z\to X$ of degree $N$. By \cite[Lemma 2.3]{Vie83}, $h_*{O}_Z(K_Z)$ contains as a direct summand the sheaf 
$$\mathcal{O}_X(D-K_X+K_X)\otimes \mathcal{O}_X(-\lfloor \frac{1}{N}\cdot N\Delta \rfloor)\cong \mathcal{O}_X(D).$$
By Grauert-Riemenschneider vanishing theorem, we have that $R^ih_*\mathcal{O}_Z(K_Z)=0$ for $i>0$. We deduce that 
$$R^i(f\circ h)_*\mathcal{O}_Z(K_Z)\cong R^if_*h_*\mathcal{O}_Z(K_Z)$$ 
contains $R^if_*\mathcal{O}_X(D)$ as a direct summand for every $i\in \N$ by Grothendieck spectral sequence. Since $R^i(f\circ h)_*\mathcal{O}_Z(K_Z)$ has the Chen-Jiang decomposition for every $i\in \N$ by \cite[Theorem A]{PPS17}, the sheaf $R^if_*\mathcal{O}_X(D)$ has the Chen-Jiang decomposition for every $i\in \N$ by Proposition \ref{sum}.
\end{proof}

Our next proposition establishes an equivalence between four statements about the positivity of pushforwards of klt pairs.

\begin{prop}\label{equi}
Let $f$ be a morphism from a klt pair $(X, \Delta)$ to an abelian variety $A$, $m\geq1$ a rational number and $D$ a Cartier divisor on $X$ such that $D\sim_{\Q}m(K_X+\Delta)$. Then the following four statements are equivalent:
\begin{enumerate}

	\item[$\mathrm{(i)}$] The sheaf $f_*\mathcal{O}_X(D)$ has the Chen-Jiang decomposition.

	\item[$\mathrm{(ii)}$] There exists an isogeny $\varphi\colon A'\to A$ such that $\varphi^*f_*\mathcal{O}_X(D)$ is globally generated.
	
	\begin{center}
	\begin{tikzcd}
			X' \arrow[r, "\varphi'"] \arrow[d, "f'"] & X \arrow[d, "f"] \\
			 A' \arrow[r, "\varphi" ] & A 
	\end{tikzcd}
	\end{center}
	
	\item[$\mathrm{(iii)}$] There exists a generically finite surjective morphism $h\colon Z\to X$ from a smooth projective variety $Z$ such that $f_*\mathcal{O}_X(D)$ is a direct summand of $(f\circ h)_*\mathcal{O}_Z(K_Z)$.
	
	\item[$\mathrm{(iv)}$] The continuous evaluation morphism $e_{f_*\mathcal{O}_X(D)}$ associated to $f_*\mathcal{O}_X(D)$ is surjective.
\end{enumerate}
\end{prop}

\begin{proof}
We can assume that $f_*\mathcal{O}_X(D)\neq0$. The fact that statement (i) implies (ii) follows from the proof of \cite[Theorem 5.1]{LPS20}. 

We now prove that statement (ii) implies (iii). We have an isogeny $\varphi\colon A'\to A$ such that $\varphi^*f_*\mathcal{O}_X(D)$ is globally generated. We define a $\Q$-divisor $\Delta'$ by $K_{X'}+\Delta'=\varphi'^*(K_X+\Delta)$. Since $\varphi'$ is an \'etale morphism, the new pair $(X', \Delta')$ is klt and $\Delta'$ is effective. Define $D'$ as $\pi^*D$ then we have $D'\sim_{\Q}m(K_{X'}+\Delta')$. By the flat base change theorem, we know that $f'_*\mathcal{O}_{X'}(D')\cong\varphi^*f_*\mathcal{O}_X(D)$ is globally generated. 

Next, we consider the following adjoint morphism 
$$f'^*f'_*\mathcal{O}_{X'}(D')\to \mathcal{O}_{X'}(D').$$
The image of this morphism is $D'\otimes\mathcal{I}$ and $\mathcal{I}$ is the relative base ideal of $D'$. We can take a log resolution $\mu\colon Y\to X'$ of $\mathcal{I}$ and $(X', \Delta')$ as in the following diagram. 
	\begin{center}
	\begin{tikzcd}
			Y \arrow[r, "\mu"] \arrow[dr, "g" swap] & X' \arrow[r, "\varphi'"] \arrow[d, "f'"] \arrow[dr] & X \arrow[d, "f"] \\
			& A' \arrow[r, "\varphi" ] & A 
	\end{tikzcd}
	\end{center}
Then we have
	$$K_Y+\Delta_Y\sim_\mathbb{Q}\mu^*(K_{X'}+\Delta')+E,$$
where the $\Q$-divisors $\Delta_Y$ and $E$ are effective and have no common components and $E$ is $\mu$-exceptional. We also have that the image of the new adjoint morphism 
$$g^*g_*\mathcal{O}_{Y}(D_Y)\to \mathcal{O}_{Y}(D_Y)$$
is a line bundle $\mathcal{O}_{Y}(D_Y-F)$ where $D_Y=\mu^*D'$ is a Cartier divisor, $F$ is an effective divisor and $\Delta_Y+E+F$ has simple normal crossings support. Since $(X', \Delta')$ is klt, we know $(Y, \Delta_Y':=\Delta_Y+\frac{\lceil mE\rceil}{m}-E)$ is also klt and  
	$$m(K_Y+\Delta_Y')\sim_\mathbb{Q}\mu^*(m(K_{X'}+\Delta'))+\lceil mE\rceil\sim_{\Q} D_Y+\lceil mE\rceil.$$
Denote $D_Y+\lceil mE\rceil$ by $G$. We have $g_*\mathcal{O}_{Y}(G)\cong f'_*\mathcal{O}_{X'}(D')$ since the effective divisor $\lceil mE\rceil$ is $\mu$-exceptional and thus $g_*\mathcal{O}_{Y}(G)$ is globally generated. The image of the adjoint morphism
$$g^*g_*\mathcal{O}_{Y}(G)\to \mathcal{O}_{Y}(G)$$
is $\mathcal{O}_{Y}(G-F-\lceil mE\rceil)$. Let $G'$ be $F+\lceil mE\rceil$. We claim that $g_*\mathcal{O}_{Y}(G-G'')\cong g_*\mathcal{O}_{Y}(G)$ for any effective divisor $G''\leq G'$. This is because we can factor the adjoint morphism as
$$g^*g_*\mathcal{O}_{Y}(G)\to \mathcal{O}_{Y}(G-G')\hookrightarrow \mathcal{O}_{Y}(G-G'')\hookrightarrow \mathcal{O}_{Y}(G)$$
and the morphism induced from pushforwards
$$g_*\mathcal{O}_{Y}(G)\to g_*\mathcal{O}_{Y}(G-G')\hookrightarrow g_*\mathcal{O}_{Y}(G-G'')\hookrightarrow g_*\mathcal{O}_{Y}(G)$$
is identity.

We claim that we can find an effective divisor $T\leq G'$ and a klt pair $(Y, N)$ such that $G-T\sim_{\Q}K_Y+N$. To do this, we use a technique from the proof of \cite[Theorem 1.7]{PS14}. Since $g_*\mathcal{O}_{Y}(G)$ is globally generated, we have $g^*g_*\mathcal{O}_{Y}(G)$ is globally generated and so is $\mathcal{O}_{Y}(G-G')$. By Bertini's theorem, we can pick a general Cartier divisor $H\sim G-G'$ which is reduced and effective such that $H$ and $\Delta_Y'+G'$ have no common components and $H+\Delta_Y'+G'$ has simple normal crossings support. We define an effective divisor $T$ as
$$T:=\bigg\lfloor \Delta_Y'+\frac{m-1}{m}G' \bigg\rfloor.$$
We have $T\leq G'$ since the coefficient of each irreducible component of $\Delta_Y'$ is smaller than $1$. We deduce
$$G-T\sim_{\Q}m(K_Y+\Delta_Y')-T=K_Y+\Delta_Y'+(m-1)(K_Y+\Delta_Y')-T$$
$$\sim_{\Q}K_Y+\Delta_Y'+\frac{m-1}{m}G-T\sim_{\Q}K_Y+\Delta_Y'+\frac{m-1}{m}(G'+H)-T$$
$$=K_Y+\frac{m-1}{m}H+\Delta_Y'+\frac{m-1}{m}G'-\bigg\lfloor \Delta_Y'+\frac{m-1}{m}G' \bigg\rfloor.$$
We define $N$ as
$$N:=\frac{m-1}{m}H+\Delta_Y'+\frac{m-1}{m}G'-\bigg\lfloor \Delta_Y'+\frac{m-1}{m}G' \bigg\rfloor.$$
The effective $\Q$-divisor $N$ has simple normal crossings support and the coefficient of each irreducible component of itself is smaller than $1$. Thus $(Y, N)$ is a klt pair and we have proved our claim.

We know that $g_*\mathcal{O}_{Y}(G-T)\cong g_*\mathcal{O}_{Y}(G)$ since $0\leq T\leq G'$. Thus we have
$$(\varphi\circ g)_*\mathcal{O}_{Y}(G-T)\cong (\varphi\circ g)_*\mathcal{O}_{Y}(G)\cong \varphi_*f'_*\mathcal{O}_{X'}(D')$$
$$\cong f_*\varphi'_*\mathcal{O}_{X'}(\varphi'^*D)\cong f_*(\mathcal{O}_{X}(D)\otimes\varphi'_*\mathcal{O}_{X'}).$$
We know that $\mathcal{O}_{X}$ is a direct summand of $\varphi'_*\mathcal{O}_{X'}$. Then $f_*(\mathcal{O}_{X}(D))$ is a direct summand of $(\varphi\circ g)_*\mathcal{O}_{Y}(G-T)$. Since $G-T\sim_{\Q}K_Y+N$, there exists a generically finite surjective morphism $h'\colon Z\to Y$ from a smooth projective variety $Z$ as in the following diagram such that $g_*\mathcal{O}_{Y}(G-T)$ is a direct summand of $(g\circ h')_*\mathcal{O}_Z(K_Z)$ by Lemma \ref{m=1}. 
	\begin{center}
	\begin{tikzcd}
		Z \arrow[r, "h'"] &Y \arrow[r, "\mu"] \arrow[dr, "g"] & X' \arrow[r, "\varphi'"] \arrow[d, "f'"] \arrow[dr] & X \arrow[d, "f"] \\
			&& A' \arrow[r, "\varphi" ] & A 
	\end{tikzcd}
	\end{center}
Then $(\varphi\circ g)_*\mathcal{O}_{Y}(G-T)$ is a direct summand of $(\varphi\circ g\circ h')_*\mathcal{O}_Z(K_Z)$ and so is $f_*(\mathcal{O}_{X}(D))$. Define the morphism $h\colon Z\to X$ as $\varphi'\circ\mu\circ h'$. We have that $\varphi\circ g\circ h'=f\circ h$ and $h$ is a generically finite surjective morphism such that $f_*(\mathcal{O}_{X}(D))$ is a direct summand of $(f\circ h)_*\mathcal{O}_Z(K_Z)$.

Next, we prove that statement (iii) implies (i). It follows from \cite[Theorem A]{PPS17} and Proposition \ref{sum}. 

Lemma \ref{tor} implies that statements (ii) and (iv) are equivalent.
\end{proof}

\begin{coro}\label{red}
Let $f$ be a morphism from a klt pair $(X, \Delta)$ to an abelian variety $A$, $m\geq1$ a rational number and $D$ a Cartier divisor on $X$ such that $D\sim_{\Q}m(K_X+\Delta)$. Let $\varphi\colon A'\to A$ be an isogeny. Then $f_*\mathcal{O}_X(D)$ has the Chen-Jiang decomposition if and only if $f'_*\mathcal{O}_{X'}(D')$ has the Chen-Jiang decomposition where $D'=\varphi'^*D$.

	\begin{center}
	\begin{tikzcd}
			X' \arrow[r, "\varphi'"] \arrow[d, "f'"] & X \arrow[d, "f"] \\
			 A' \arrow[r, "\varphi" ] & A 
	\end{tikzcd}
	\end{center}
	
\end{coro}

\begin{proof}
We first prove the if part. By Proposition \ref{equi}, there exists an isogeny $\psi: A''\to A'$ such that $\psi^*f'_*\mathcal{O}_{X'}(D')\cong \psi^*\varphi^*f_*\mathcal{O}_X(D)$ is globally generated as in the following diagram. 
	\begin{center}
	\begin{tikzcd}
			X''\arrow[r,"\psi'"]\arrow[d, "f''"] & X' \arrow[r, "\varphi'"] \arrow[d, "f'"] & X \arrow[d, "f"] \\
			 A''\arrow[r, "\psi" ] &A' \arrow[r, "\varphi" ] & A 
	\end{tikzcd}
	\end{center}
By Proposition \ref{equi} again, $f_*\mathcal{O}_X(D)$ has the Chen-Jiang decomposition.

Next, we prove the only if part. By Proposition \ref{equi}, there exists an isogeny $g\colon B\to A$ such that $g^*f_*\mathcal{O}_X(D)$ is globally generated. Consider the following diagram for the fiber product $B':=B\times_{A}A'$. 
	\begin{center}
	\begin{tikzcd}
			B' \arrow[r, "p"] \arrow[d, "q"] & A' \arrow[d, "\varphi"] \\
			 B \arrow[r, "g" ] & A 
	\end{tikzcd}
	\end{center}
The morphisms $p$ and $q$ are the projections and isogenies between abelian varieties since $g$ and $\varphi$ are isogenies. Since 
$$p^*f'_*\mathcal{O}_{X'}(D')\cong p^*\varphi^*f_*\mathcal{O}_X(D)\cong q^*g^*f_*\mathcal{O}_X(D)$$ 
is globally generated, $f'_*\mathcal{O}_{X'}(D')$ has the Chen-Jiang decomposition by Proposition \ref{equi}.
\end{proof}

The following proposition claims that we only need to find a positive integer $N$ such that $f_*\mathcal{O}_X(ND)$ has the Chen-Jiang decomposition in order to prove that $f_*\mathcal{O}_X(D)$ has the same property.

\begin{prop}\label{N}
Let $f$ be a morphism from a klt pair $(X, \Delta)$ to an abelian variety $A$, $m\geq1$ a rational number and $D$ a Cartier divisor on $X$ such that $D\sim_{\Q}m(K_X+\Delta)$. If there exists a positive integer $N$ such that $f_*\mathcal{O}_X(ND)$ has the Chen-Jiang decomposition, then $f_*\mathcal{O}_X(D)$ has the Chen-Jiang decomposition.
\end{prop}

\begin{proof}
If the sheaf $f_*\mathcal{O}_X(ND)$ is $0$, then $f_*\mathcal{O}_X(D)=0$ and the statement is trivial. Thus we can assume $f_*\mathcal{O}_X(ND)\neq0$. By Proposition \ref{equi} and Corollary \ref{red}, we can assume $f_*\mathcal{O}_X(ND)$ is globally generated. 

We use a similar method as in the proof of Proposition \ref{equi}. We consider the following two adjoint morphisms 
$$f^*f_*\mathcal{O}_{X}(D)\to \mathcal{O}_{X}(D)\quad \text{and} \quad f^*f_*\mathcal{O}_{X}(ND)\to \mathcal{O}_{X}(ND).$$
We can take a log resolution $\mu\colon Y\to X$ of $(X, \Delta)$ and the relative base ideals of $D$ and $ND$.
	\begin{center}
	\begin{tikzcd}
			Y \arrow[r, "\mu"] \arrow[dr, "g" swap] & X \arrow[d, "f"] \\
			& A 
	\end{tikzcd}
	\end{center}
Then we have
	$$K_Y+\Delta_Y\sim_\mathbb{Q}\mu^*(K_{X}+\Delta)+E,$$
where the $\Q$-divisors $\Delta_Y$ and $E$ are effective and have no common components and $E$ is $\mu$-exceptional. We also have that the images of the new adjoint morphisms 
$$g^*g_*\mathcal{O}_{Y}(D_Y)\to \mathcal{O}_{Y}(D_Y)\quad \text{and} \quad g^*g_*\mathcal{O}_{Y}(ND_Y)\to \mathcal{O}_{Y}(ND_Y)$$
are line bundles $\mathcal{O}_{Y}(D_Y-F)$ and $\mathcal{O}_{Y}(ND_Y-F_N)$ where $D_Y=\mu^*D$ is a Cartier divisor, $F$ and $F_N$ are effective divisors and $\Delta_Y+E+F+F_N$ has simple normal crossings support. Since $(X, \Delta)$ is klt, we know $(Y, \Delta_Y':=\Delta_Y+\frac{\lceil mE\rceil}{m}-E)$ is also klt and  
$$m(K_Y+\Delta_Y')\sim_\mathbb{Q}\mu^*(m(K_{X}+\Delta))+\lceil mE\rceil\sim_{\Q} D_Y+\lceil mE\rceil.$$
Denote $D_Y+\lceil mE\rceil$ by $G$. We have $g_*\mathcal{O}_{Y}(G)\cong f_*\mathcal{O}_{X}(D)$ and $g_*\mathcal{O}_{Y}(NG)\cong f_*\mathcal{O}_{X}(ND)$ since $\lceil mE\rceil$ is $\mu$-exceptional and effective. The images of the two adjoint morphisms
$$g^*g_*\mathcal{O}_{Y}(G)\to \mathcal{O}_{Y}(G)\quad \text{and} \quad g^*g_*\mathcal{O}_{Y}(NG)\to \mathcal{O}_{Y}(NG)$$
are $\mathcal{O}_{Y}(G-F-\lceil mE\rceil)$ and $\mathcal{O}_{Y}(NG-F_N-N\lceil mE\rceil)$. Let $G'$ be the divisor $F+\lceil mE\rceil$ and $G'_N$ the divisor $F_N+N\lceil mE\rceil$. We claim that $G'_N\leq NG'$. We have the natural morphism 
$$(g_*\mathcal{O}_{Y}(G))^{\otimes N}\to g_*\mathcal{O}_{Y}(NG)$$
and the following commutative diagram.
	\begin{center}
	\begin{tikzcd}
			(g^*g_*\mathcal{O}_{Y}(G))^{\otimes N}\arrow[r] \arrow[d] & g^*g_*\mathcal{O}_{Y}(NG) \arrow[d] \\
			 (\mathcal{O}_{Y}(G))^{\otimes N} \arrow[r] & \mathcal{O}_{Y}(NG) 
	\end{tikzcd}
	\end{center}
We deduce from the diagram that $\mathcal{O}_{Y}(N(G-G'))$ is contained in $\mathcal{O}_{Y}(NG-G'_N)$ and thus $\mathcal{O}_{Y}(-NG')$ is contained in $\mathcal{O}_{Y}(-G'_N)$ as subsheaves of $\mathcal{O}_{Y}$ which proves the claim.

We know $g_*\mathcal{O}_{Y}(NG)$ is globally generated and thus $\mathcal{O}_{Y}(NG-G'_N)$ is globally generated. By Bertini's theorem, we can pick a general Cartier divisor $H\sim NG-G'_N$ which is reduced and effective such that $H$ and $\Delta_Y'+G'+G'_N$ have no common components and $H+\Delta_Y'+G'+G'_N$ has simple normal crossings support. We define an effective divisor $T$ as
$$T:=\bigg\lfloor \Delta_Y'+\frac{m-1}{Nm}G'_N \bigg\rfloor.$$
We have
$$T\leq\bigg\lfloor \Delta_Y'+\frac{m-1}{m}G' \bigg\rfloor\leq G'$$
since $G'_N\leq NG'$ and the coefficient of each irreducible component of $\Delta_Y'$ is smaller than $1$. Since $0\leq T\leq G'$, we know $g_*\mathcal{O}_{Y}(G)\cong g_*\mathcal{O}_{Y}(G-T)$ by the same argument of Proposition \ref{equi}. We have
$$G-T\sim_{\Q}m(K_Y+\Delta_Y')-T=K_Y+\Delta_Y'+(m-1)(K_Y+\Delta_Y')-T$$
$$\sim_{\Q}K_Y+\Delta_Y'+\frac{m-1}{m}G-T\sim_{\Q}K_Y+\Delta_Y'+\frac{m-1}{Nm}(G'_N+H)-T$$
$$=K_Y+\frac{m-1}{Nm}H+\Delta_Y'+\frac{m-1}{Nm}G'_N-\bigg\lfloor \Delta_Y'+\frac{m-1}{Nm}G'_N \bigg\rfloor.$$
We define $N$ as
$$N:=\frac{m-1}{Nm}H+\Delta_Y'+\frac{m-1}{Nm}G'_N-\bigg\lfloor \Delta_Y'+\frac{m-1}{Nm}G'_N \bigg\rfloor.$$
The effective $\Q$-divisor $N$ has simple normal crossings support and the coefficient of each irreducible component of itself is smaller than $1$. Thus $(Y, N)$ is a klt pair. Since $G-T\sim_{\Q}K_Y+N$, the sheaf $f_*\mathcal{O}_{X}(D)\cong g_*\mathcal{O}_{Y}(G)\cong g_*\mathcal{O}_{Y}(G-T)$ has the Chen-Jiang decomposition by Lemma \ref{m=1}.
\end{proof}

We will prove our main theorems by induction on $\dim A$. Our next lemma treats the central part of the induction. To prove it, we follow the strategy used in \cite[Sections 8 and 9]{LPS20} together with some additional techniques.

\begin{prop}\label{SNC}
Assume that Theorem \ref{main3} is true when the base abelian variety is of dimension up to $n-1$. Let $f$ be a morphism from a log smooth klt pair $(X, \Delta)$ to an abelian variety $A$ of dimension $n$. Then $f_*\mathcal{O}_X(N(K_X+\Delta))$ has the Chen-Jiang decomposition for every positive integer $N$ which is sufficiently big and divisible.
\end{prop}

\begin{proof}
We choose $N$ sufficiently divisible such that $N(K_X+\Delta)$ is Cartier. We can assume that $f_*\mathcal{O}_X(N(K_X+\Delta))\neq0$ for every positive integer $N$ which is sufficiently big and divisible. We choose $N$ sufficiently big and divisible such that Lemma \ref{split} works for $f_*\mathcal{O}_X(N(K_X+\Delta))$. 

We denote $f_*\mathcal{O}_X(N(K_X+\Delta))$ by $\mathcal{F}$ and $N(K_X+\Delta)$ by $D$. The sheaf $\mathcal{F}$ is a GV-sheaf by \cite[Variant 5.5]{PS14}. We consider the continuous evaluation morphism associated to $\mathcal{F}$
$$e_{\mathcal{F}}\colon \bigoplus_{\alpha\in \Pic^0(A)\atop \mathrm{torsion}}H^0(A, \mathcal{F}\otimes\alpha)\otimes\alpha^{-1}\to \mathcal{F}$$
and the sheaf $\mathcal{G}:=\Image e_{\mathcal{F}}\subseteq \mathcal{F}$. Lemma \ref{split} will work for the new pair for the same $N$ if we take an isogeny on $A$ and do the base change by the discussion after it. Thus we can assume that $\mathcal{G}$ is globally generated by Lemma \ref{tor} and Corollary \ref{red}. By Proposition \ref{equi} we only need to prove that $\mathcal{G}=\mathcal{F}$ to deduce our lemma. We consider the adjoint morphism
$$f^*\mathcal{G}\to\mathcal{O}_X(N(K_X+\Delta)).$$
By taking a log resolution, we can assume the image of the adjoint morphism is of the form $\mathcal{O}_X(N(K_X+\Delta)-E)$ where $E$ is an effective divisor and that $E+\Delta$ has simple normal crossings support. Since $\mathcal{G}$ is globally generated, the line bundle $\mathcal{O}_X(N(K_X+\Delta)-E)$ is globally generated and thus there exists an effective and reduced divisor $H\sim N(K_X+\Delta)-E$ such that $H$ and $E+\Delta$ have no common components and $H+E+\Delta$ has simple normal crossings support. By the same argument of Proposition \ref{equi}, there exists an effective divisor $T\leq E$ such that $D-T\sim_{\Q}K_X+\Delta'$ where $(X, \Delta')$ is a log smooth klt pair. We claim that the inclusion
$$\mathcal{G}\hookrightarrow f_*\mathcal{O}_X(D-T)$$
is an identity. Since $D-T\sim_{\Q}K_X+\Delta'$, the continuous evaluation morphism associated to $f_*\mathcal{O}_X(D-T)$ is surjective by Lemma \ref{m=1} and Proposition \ref{equi}. We know that
$$H^0(A, \mathcal{G}\otimes\alpha)\subseteq H^0(A, f_*\mathcal{O}_X(D-T)\otimes\alpha)\subseteq H^0(A, f_*\mathcal{O}_X(D)\otimes\alpha)$$
for every torsion line bundle $\alpha\in\Pic^0(A)$ and the two section spaces on the left and right are equal from the definition of $\mathcal{G}$. Thus we deduce that 
$$\mathcal{G}=f_*\mathcal{O}_X(D-T)$$
since the continuous evaluation morphisms associated to $\mathcal{G}$ and $f_*\mathcal{O}_X(D-T)$ are surjective. By Lemma \ref{m=1}, there exists a generically finite surjective morphism $h\colon W\to X$ from a smooth projective variety $W$ such that $\mathcal{G}=f_*\mathcal{O}_X(D-T)$ is a direct summand of $(f\circ h)_*\mathcal{O}_W(K_W)$. Thus $\mathcal{G}$ is a GV-sheaf by \cite{GL87} and \cite{Hac04} and the cohomological support loci $V_l^i(A, \mathcal{G})$ are finite unions of torsion subvarieties of $\Pic^0(A)$ for every $i$ and $l$ by \cite{GL91}, \cite{Sim93} and \cite[Lemma 10.3]{HPS18}. Thus the sheaf $\mathcal{H}:=\mathcal{F}/\mathcal{G}$ is also a GV-sheaf. We only need to prove $\mathcal{H}=0$ or $V^0(A, \mathcal{H})$ is empty to deduce our lemma. By \cite[Theorem 1.3]{Shi16}, the cohomological support loci $V_l^0(A, \mathcal{F})$ are finite unions of torsion subvarieties of $\Pic^0(A)$ for every $l$. We deduce that
$$H^0(A, \mathcal{G}\otimes\alpha)=H^0(A, \mathcal{F}\otimes\alpha)$$
for every $\alpha\in \Pic^0(A)$ since it is true when $\alpha$ is a torsion point and torsion points are dense inside $V_l^0(A, \mathcal{G})$ and $V_l^0(A, \mathcal{F})$ by the structures of these sets. By the exact sequence
$$0\to H^0(A, \mathcal{G}\otimes\alpha)\to H^0(A, \mathcal{F}\otimes\alpha)\to H^0(A, \mathcal{H}\otimes\alpha)\to H^1(A, \mathcal{G}\otimes\alpha),$$
we deduce that $V^0(A, \mathcal{H})\subseteq V^1(A, \mathcal{G})$. Thus $V^0(A, \mathcal{H})$ is contained in a finite union of torsion subvarieties of codimension $\geq1$. We can choose an isogeny and do the base change using this isogeny freely by Corollary \ref{red}. Thus we can assume that all the irreducible components of $V^1(A, \mathcal{G})$ are abelian varieties. If $V^1(A, \mathcal{G})$ is empty, then $\mathcal{H}=0$. Assume that $V^1(A, \mathcal{G})$ is not empty. If $V^1(A, \mathcal{G})$ does not contain any positive-dimensional abelian subvariety, then $V^1(A, \mathcal{G})=\{\mathcal{O}_A\}$. Then $\mathcal{H}=0$ follows from Lemma \ref{split} and the same proof of \cite[Proposition 8.3]{LPS20}.

We now assume that $V^1(A, \mathcal{G})$ contains a positive-dimensional abelian subvariety. If $V^0(A, \mathcal{H})$ does not intersect any of these positive-dimensional abelian subvarieties, then $V^0(A, \mathcal{H})$ is empty and thus $\mathcal{H}=0$. We assume that $V^0(A, \mathcal{H})$ intersects at least one of them which is denoted by $C\subseteq V^1(A, \mathcal{G})$. Then the abelian subvariety $C$ is the image of the pullback morphism $p^*\colon\Pic^0(B)\to\Pic^0(A)$ where $p\colon A\to B$ is a fibration to an abelian variety $B$. Since $\mathcal{G}$ is a GV-sheaf, we have $1\leq\dim B\leq n-1$. After a base change by an isogeny on $B$, we can assume $A=B\times F$ where $B$ and $F$ are abelian varieties, $p$ is the first projection and $q$ is the second projection. Denote $p\circ f$ by $g$.	
	\begin{center}
	\begin{tikzcd}
			X \arrow[r, "f"] \arrow[rr, bend right, "g"] & B\times F \arrow[r, "p"] & B
	\end{tikzcd}
	\end{center}
We know the continuous evaluation morphism $e_{p_*\mathcal{G}}$ associated to $p_*\mathcal{G}$ is surjective since $p_*\mathcal{G}$ is a direct summand of $(g\circ h)_*\mathcal{O}_W(K_W)$. Since we assume that Theorem \ref{main3} is true when the base abelian variety is of dimension up to $n-1$, the continuous evaluation morphism $e_{p_*\mathcal{F}}$ associated to $p_*\mathcal{F}$ is surjective by Proposition \ref{equi} and thus $p_*\mathcal{G}=p_*\mathcal{F}$ since we have
$$H^0(B, p_*\mathcal{G}\otimes\beta)=H^0(B, p_*\mathcal{F}\otimes\beta)$$
for every $\beta\in\Pic^0(B)$. By a similar argument, we can prove that 
$$p_*(\mathcal{G}\otimes\alpha)=p_*(\mathcal{F}\otimes\alpha)$$
for every torsion point $\alpha\in\Pic^0(A)$.

We now consider $Z=g(X)$ which is a reduced and irreducible subvariety of $B$ and $\mathcal{G}$, $\mathcal{F}$ and $\mathcal{H}$ as coherent sheaves on $f(X)\subseteq Z\times F$. For any $b\in Z$, let $X_b=g^{-1}(b)$ and $f_b$ the induced morphism as in the following base change diagram.
	\begin{center}
	\begin{tikzcd}
			X_b \arrow[r, hook, "h_b"] \arrow[d, "f_b"] & X \arrow[d, "f"] \arrow[dd, bend left=60, "g"] \\
			 F \arrow[r, hook, "k_b"] \arrow[d, ""] & Z\times F \arrow[d, "p"] \\
			 \{b\} \arrow[r, hook, "i_b" ] & Z
	\end{tikzcd}
	\end{center}
The morphism $k_b$ is $(i_b, \id_F)$. By the generic smoothness theorem, there exists a nonempty open subset $U'\subseteq Z$ such that $g$ is smooth over $U'$ and $(X_b, \Delta_b)$ is a log smooth klt pair for every $b\in U'$ where $\Delta_b=\Delta|_{X_b}$. Denote $(f_{b})_*\mathcal{O}_{X_b}(N(K_{X_b}+\Delta_b))$ by $\mathcal{F}_b$. By \cite[Proposition 4.1]{LPS20}, we can shrink $U'$ such that the base change morphism
$$k_b^*\mathcal{F}\to\mathcal{F}_b$$
is an isomorphism for every $b\in U'$ and thus the base change morphism
$$k_b^*(\mathcal{F}\otimes q^*\gamma)\to\mathcal{F}_b\otimes\gamma$$ 
is an isomorphism for every $b\in U'$ and every torsion point $\gamma\in\Pic^0(F)$ where $q$ is the second projection onto $F$. Since we assume that Theorem \ref{main3} is true when the base abelian variety is of dimension up to $n-1$ and $1\leq\dim F\leq n-1$, the continuous evaluation morphism associated to $\mathcal{F}_b$
$$e_{\mathcal{F}_b}\colon \bigoplus_{\gamma\in \Pic^0(F)\atop \mathrm{torsion}}H^0(F, \mathcal{F}_b\otimes\gamma)\otimes\gamma^{-1}\to\mathcal{F}_b$$
is surjective for every $b\in U'$. For every torsion point $\gamma\in\Pic^0(F)$, there exists a nonempty open subset $U_{\gamma}\subseteq U'$ such that the base change morphism
$$i_b^*\Big(g_*\big(\mathcal{O}_X(N(K_X+\Delta))\otimes f^*q^*\gamma\big)\Big)\to H^0(X_b, \mathcal{O}_{X_b}(N(K_{X_b}+\Delta_b))\otimes f_b^*\gamma)$$
is surjective for every $b\in U_{\gamma}$ by the base change theorem. Thus we know 
$$i_b^*(p_*(\mathcal{F}\otimes q^*\gamma))\to H^0(F, \mathcal{F}_b\otimes\gamma)$$
is surjective for every $b\in U_{\gamma}$ by the projection formula. Consider the set
$$U:=\bigcap_{\gamma\in\Pic^0(F)\atop\mathrm{torsion}}U_{\gamma}$$
which is an intersection of countably many nonempty open subsets of $Z$ since $\Pic^0(F)$ only has countably many torsion points. The set $U$ is not empty since an irreducible variety over an uncountable and algebraically closed field cannot be a countable union of proper closed subsets. By the surjective morphisms above, we deduce that the morphism $e$ induced from the adjoint morphisms
$$e\colon\bigoplus_{\gamma\in \Pic^0(F)\atop \mathrm{torsion}}p^*p_*(\mathcal{F}\otimes q^*\gamma)\otimes q^*\gamma^{-1}\to\mathcal{F}$$
is surjective after restricted to the fiber $p^{-1}(b)$ for every $b\in U$ since $e_{\mathcal{F}_b}$ is surjective for every $b\in U$. Then $e$ is surjective at every point of the fiber $p^{-1}(b)$ for every $b\in U$ by Nakayama's lemma since $\mathcal{F}$ is coherent. Since
$$p_*(\mathcal{G}\otimes q^*\gamma)=p_*(\mathcal{F}\otimes q^*\gamma)$$
for every torsion point $\gamma\in\Pic^0(F)$, the morphism $e$ factors through the subsheaf $\mathcal{G}$ of $\mathcal{F}$. We deduce that $\mathcal{G}=\mathcal{F}$ on $p^{-1}(U)$ and $\Supp\mathcal{H}$ does not intersect the set $p^{-1}(U)$. Denote $Z\times F\setminus p^{-1}(U_{\gamma})$ by $V_{\gamma}$ which is a proper closed subset of $Z\times F$. Then we have
$$\Supp\mathcal{H}\subseteq \bigcup_{\gamma\in\Pic^0(F)\atop\mathrm{torsion}}V_{\gamma}.$$
Since $\mathcal{H}$ is coherent, $\Supp\mathcal{H}$ is closed and can be decomposed as a union of irreducible components
$$\Supp\mathcal{H}=\bigcup_{k\in K}Z_k$$
where $K$ is a finite index set. We deduce that
$$Z_k=\bigcup_{\gamma\in\Pic^0(F)\atop\mathrm{torsion}}(Z_k\cap V_{\gamma}).$$
Since an irreducible variety over an uncountable and algebraically closed field cannot be a countable union of proper closed subsets, we deduce that $Z_k\subseteq V_{\gamma_k}$ for some torsion point $\gamma_k\in\Pic^0(F)$. Thus $\Supp\mathcal{H}$ does not intersect the nonempty open set
$$p^{-1}(\bigcap_{k\in K}U_{\gamma_k})$$
and we deduce that $p_*\mathcal{H}$ is a torsion sheaf on $Z$. We have the exact sequence
$$0\to p_*\mathcal{G}\to p_*\mathcal{F}\to p_*\mathcal{H}\to R^1p_*\mathcal{G}.$$
We know $\mathcal{G}$ is a direct summand of $(f\circ h)_*\mathcal{O}_W(K_W)$. By the five-term exact sequence, we deduce that $R^1p_*((f\circ h)_*\mathcal{O}_W(K_W))$ is a subsheaf of $R^1(p\circ f\circ h)_*\mathcal{O}_W(K_W)$ which is a torsion-free sheaf by \cite[Theorem 2.1]{Kol86}. Since $p_*\mathcal{G}=p_*\mathcal{F}$, the torsion sheaf $p_*\mathcal{H}$ is a subsheaf of $R^1p_*\mathcal{G}$ which is a torsion-free sheaf. Thus we deduce that $p_*\mathcal{H}=0$ and
$$H^0(A, \mathcal{H}\otimes p^*\beta)\cong H^0(B, p_*\mathcal{H}\otimes\beta)=0$$
for every $\beta\in\Pic^0(B)$. Thus $V^0(A, \mathcal{H})$ does not intersect $C$ which is a contradiction and we finish our proof.
\end{proof}

\begin{rem}\label{genera}
The invariance of plurigenera for smooth families of smooth varieties is used in the proof of \cite[Lemma 9.2]{LPS20} to conclude that there exists a nonempty open subset $O\subseteq Z$ such that $\Supp\mathcal{H}$ does not intersect $O\times F$. Our argument in Proposition \ref{SNC} shows that we can prove the same result even in the case of klt pairs without appealing to the invariance of plurigenera.
\end{rem}

Next, we prove our main theorems.

\begin{proof}[Proof of Theorem \ref{main3}]
We prove the theorem by induction on $\dim A$. The existence of the Chen-Jiang decomposition follows from Propositions \ref{N} and \ref{SNC} by taking a log resolution. The statement that we can choose $\alpha_i$ which becomes trivial when pulled back by the isogeny $\varphi$ in Theorem \ref{main1} follows from the same proof of \cite[Theorem C]{LPS20}.
\end{proof}

\begin{proof}[Proof of Theorem \ref{main1}]
For an individual $l$, the existence of such an isogeny follows from Theorem \ref{main3} and Proposition \ref{equi}. Next we prove that we can find an isogeny which works for every $l$. By \cite[Theorem 1.2]{BCHM}, for $N\in\N$ sufficiently divisible, the $\mathcal{O}_A$-algebra
$$\mathfrak{R}(f, ND)=\bigoplus_{i\in\N}f_*\mathcal{O}_X(iND)$$
is finitely generated. Thus the $\mathcal{O}_A$-algebra $\mathfrak{R}(f, D)$ is finitely generated by a Theorem of E. Noether on the finiteness of integral closure. Thus we can find an isogeny which works for every $l$.
\end{proof}

\begin{proof}[Proof of Theorem \ref{main2}]
It follows from Theorem \ref{main3} and Proposition \ref{equi}.
\end{proof}

As mentioned in the introduction, for the sake of completeness we also give an alternative approach to the main theorems, which uses the minimal model program as in the following lemma. However, we need an extra assumption which is that the general fiber $(F, \Delta|_F)$ of $f$ has a good minimal model. The good minimal model conjecture is a famous conjecture in birational geometry which states that every klt pair $(X, \Delta)$ has a good minimal model if $K_X+\Delta$ is pseudoeffective. The conjecture is known to hold in dimensions up to $3$ by a lot of work and when the pair $(X, \Delta)$ is of log general type by \cite{BCHM}.

\begin{prop}\label{MMP}
Let $f$ be a morphism from a klt pair $(X, \Delta)$ to an abelian variety $A$, $m>0$ a rational number and $D$ a Cartier divisor on $X$ such that $D\sim_{\Q}m(K_X+\Delta)$. If the general fiber $(F, \Delta|_F)$ of $f$ has a good minimal model, then $f_*\mathcal{O}_X(ND)$ has the Chen-Jiang decomposition for every positive integer $N$ which is sufficiently big and divisible.
\end{prop}

\begin{proof}
The Stein factorization of $f$ gives a decomposition of $f$ as $g\circ h$ where $h\colon X\to Z$ is a fibration to a normal projective variety $Z$ and $g\colon Z\to A$ is a finite morphism. By assumption, the general fiber $(F, \Delta|_F)$ of $h$ has a good minimal model and $(X, \Delta)$ is klt. Thus by \cite[Theorem 1.2]{BCHM} and \cite[Theorem 2.12]{HX13}, $(X, \Delta)$ has a good minimal model $(Y, \Delta_Y)$ over $Z$. Let $\xi\colon X\dasharrow Y$ be the birational contraction and $(p,q)\colon W\to X\times Y$ a smooth resolution of indeterminacies of $\xi$ as in the following diagram.
	
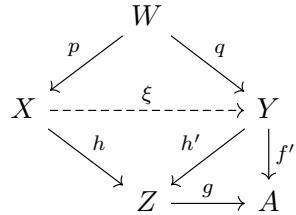
\begin{center}
	\begin{tikzcd}
			& W \arrow[dl, "p" swap] \arrow[dr, "q"] & \\
			X \arrow[rr, dashed, "\xi" ] \arrow[dr, "h"]&& Y \arrow[dl, "h'" swap]\arrow[d, "f'"]\\&Z \arrow[r, "g"]&A
	\end{tikzcd}
	\end{center} 
The pair $(Y, \Delta_Y)$ is klt and we have 
$$	p^*(K_X+\Delta)\sim_\Q q^*(K_Y+\Delta_Y)+E,$$
where $E$ is an effective $q$-exceptional $\Q$-divisor. We can choose a positive integer $N$ sufficiently divisible such that $Nm(K_X+\Delta)$, $Nm(K_Y+\Delta_Y)$ and $NmE$ are Cartier divisors, $ND\sim Nm(K_X+\Delta)$ and
$$p^*(Nm(K_X+\Delta))\sim q^*(Nm(K_Y+\Delta_Y))+NmE.$$
Note that $m>0$ is a rational number. We can also make $N$ sufficiently big such that $Nm\geq 1$. We deduce that
$$f_*\mathcal{O}_X(ND)\cong f_*\mathcal{O}_X(Nm(K_X+\Delta))\cong f_*p_*\mathcal{O}_W\big(p^*(Nm(K_X+\Delta))\big)$$
$$\cong f'_*q_*\mathcal{O}_W\big(q^*(Nm(K_Y+\Delta_Y))+NmE\big)\cong f'_*\mathcal{O}_Y(Nm(K_Y+\Delta_Y))$$
since $NmE$ is an effective $q$-exceptional divisor. Since $g$ is a finite morphism and $K_Y+\Delta_Y$ is semiample over $Z$, we deduce that $K_Y+\Delta_Y$ is semiample over $A$. By \cite[Corollary 1.2]{Hu16}, $K_Y+\Delta_Y$ is semiample and thus $Nm(K_Y+\Delta_Y)$ is semiample. Thus we can find a positive integer $M$ which is sufficiently big such that  $MNm(K_Y+\Delta_Y)\sim H$ where $H$ is a reduced and effective divisor such that $(Y,\Delta_Y+\frac{1}{M}H)$ is still klt. We have that
$$Nm(K_Y+\Delta_Y)\sim_{\Q}K_Y+\Delta_Y+\frac{Nm-1}{MNm}H.$$ 
The pair $(Y, \Delta_Y+\frac{Nm-1}{MNm}H)$ is klt since $Nm\geq 1$. Thus $f'_*\mathcal{O}_Y(Nm(K_Y+\Delta_Y))$ has the Chen-Jiang decomposition by Lemma \ref{m=1}. Then $f_*\mathcal{O}_X(ND)$ has the Chen-Jiang decomposition.
\end{proof}

\begin{rem}\label{not}
The proof of Proposition \ref{MMP} is not purely algebraic since \cite[Corollary 1.2]{Hu16} builds on the result by \cite{BC15} which exploits the extension theorem from \cite{DHP13}. The proof of the extension theorem needs applications of the theory of singular Hermitian metrics.
\end{rem}

Next, we give a simple corollary of Theorems \ref{main1}, \ref{main2} and \ref{main3} regarding some special dlt pairs.

\begin{coro}
Let $f$ be a morphism from a $\Q$-factorial dlt pair $(X, \Delta)$ to an abelian variety $A$, $m\geq1$ a rational number, $H$ an ample $\Q$-Cartier $\Q$-divisor on $X$ and $D$ a Cartier divisor on $X$ such that $D\sim_{\Q}m(K_X+\Delta+H)$. Then the conclusions of Theorems \ref{main1}, \ref{main2} and \ref{main3} are true.
\end{coro}

\begin{proof}
Denote $\lfloor\Delta\rfloor$ by $S$. Since $H$ is ample, we can choose a rational number $\varepsilon>0$ which is sufficiently small such that $H':=H+\varepsilon S$ is an ample $\Q$-Cartier $\Q$-divisor. Since $(X, \Delta)$ is dlt, the pair $(X, \Delta-\varepsilon S)$ is klt and thus we can choose an effective $\Q$-Cartier $\Q$-divisor $E\sim_{\Q}H'$ such that $(X, \Delta':=\Delta-\varepsilon S+E)$ is klt. We have that
$$D\sim_{\Q}m(K_X+\Delta+H)\sim_{\Q}m(K_X+\Delta-\varepsilon S+H')\sim_{\Q}m(K_X+\Delta')$$
and this finishes the proof.
\end{proof}

\section{Applications}\label{4}

We start with a corollary of Theorem \ref{main2}. Note that \cite{PS14} and \cite{Shi16} have some results related to statements (i) and (ii) of Corollary \ref{coro1} under different assumptions.

\begin{coro}\label{coro1}
Let $f$ be a morphism from a klt pair $(X, \Delta)$ to an abelian variety $A$, $m\geq1$ a rational number and $D$ a Cartier divisor on $X$ such that $D\sim_{\Q}m(K_X+\Delta)$. Then:
\begin{enumerate}
	\item[$\mathrm{(i)}$] The sheaf $f_*\mathcal{O}_X(D)$ is a $\GV$-sheaf.

	\item[$\mathrm{(ii)}$] The cohomological support loci $V_l^i(A, f_*\mathcal{O}_X(D))$ are finite unions of torsion subvarieties of $\Pic^0(A)$ for every $i$ and $l$. 

	\item[$\mathrm{(iii)}$] The Fourier-Mukai transform $\bold{R}\hat{\mathcal{S}}\bold{R}\mathcal{H}om_{\mathcal{O}_A}(f_*\mathcal{O}_X(D), \mathcal{O}_A)$ is computed locally around each point by a linear complex of trivial vector bundles.
\end{enumerate}
\end{coro}

\begin{proof}
By Theorem \ref{main2}, statement (i) follows from \cite{GL87} and \cite{Hac04}, (ii) follows from \cite{GL91}, \cite{Sim93} and \cite[Lemma 10.3]{HPS18} and (iii) follows from \cite{CH02b} and \cite{LPS11} which are based on \cite{GL91}.
\end{proof}

Next we prove Theorem \ref{main6}. Given a klt pair $(X, \Delta)$ and a Cartier divisor $D\sim_{\Q}m(K_X+\Delta)$ on $X$ such that $\kappa(X, D)\geq0$ where $m>1$ is a rational number, we consider the Iitaka fibration associated to $D$. After a birational modification of $X$, we may assume from the beginning that the Iitaka fibration is a morphism $f\colon X\to Y$ where $Y$ is a smooth projective variety of dimension $\kappa(X, D)$. Note that we do not assume $(X, \Delta)$ is log smooth. Since $(X, \Delta)$ is klt, $X$ is of rational singularities. Thus its Albanese variety coincides with the Albanese variety of any of its log resolution by \cite[Proposition 2.3]{Rei83} and \cite[Lemma 8.1]{Kaw85a}. Thus by the universal property of the Albanese variety, we have the following commutative diagram
	\begin{center}
	\begin{tikzcd}
			X \arrow[r, "a_X"] \arrow[d, "f"] \arrow[dr, "g"]& \Alb(X) \arrow[d, "a_f"] \\
			 Y \arrow[r, "a_Y" ] & \Alb(Y) 
	\end{tikzcd}
	\end{center}
where $\Alb(X)$ and $\Alb(Y)$ are the two Albanese varieties and $a_X$ and $a_Y$ are the two Albanese morphisms. Denote $a_Y\circ f$ by $g$. We will give an explicit description of the cohomological support locus
$$V^0(X, \mathcal{O}_X(D))=\{\alpha\in\Pic^0(X)\mid\dim H^0(X, \mathcal{O}_X(D)\otimes\alpha)\geq 1\}$$
using the Iitaka fibration $f$. The morphism $f^*\colon \Pic^0(Y)\to \Pic^0(X)$ is injective since $f$ is a fibration.

\begin{proof}[Proof of Theorem \ref{main6}]
We claim that it suffices to prove the theorem when $(X, \Delta)$ is log smooth. We take a log resolution of $(X, \Delta)$ denoted by $\mu\colon W\to X$. Then we have
$$K_W+\Delta_W\sim_\mathbb{Q}\mu^*(K_{X}+\Delta)+E,$$
where the $\Q$-divisors $\Delta_W$ and $E$ are effective and have no common components, $E$ is $\mu$-exceptional and $\Delta_W+E$ has simple normal crossings support. We know $(W, \Delta_W':=\Delta_W+\frac{\lceil mE\rceil}{m}-E)$ is log smooth and klt since $m>1$ and  
$$m(K_W+\Delta_W')\sim_\mathbb{Q}\mu^*(m(K_{X}+\Delta))+\lceil mE\rceil\sim_{\Q} \mu^*D+\lceil mE\rceil.$$
By \cite[Lemma 8.1]{Kaw85a},
$$\mu^*\colon\Pic^0(X)\to \Pic^0(W)$$
is an isomorphism since $X$ is of rational singularities. Thus it suffices to prove the theorem when $(X, \Delta)$ is log smooth by considering $\mu^*D+\lceil mE\rceil$ since $\lceil mE\rceil$ is effective and $\mu$-exceptional.

We prove statement (i) first. We use $\alpha'$ to denote the Cartier divisor corresponding to the line bundle $\alpha$. We fix the torsion point $\alpha$. Given a very ample divisor $H'$ on $\Alb(Y)$, we can find a positive integer $c$ which is sufficiently big such that $cD\sim g^*H'+B'$ where $B'$ is an effective divisor on $X$ since $g$ factors through the Iitaka fibration associated to $D$. We have that $c(D+g^*L)\sim g^*(H'+cL)+B'$ where the divisor $H'+cL$ is ample. Thus we can find a positive integer $d$ which is sufficiently big and divided by $c$ such that
$$d(D+g^*L)\sim g^*H+B$$
where $B$ is an effective divisor on $X$ and $H$ is a very ample divisor on $\Alb(Y)$.

We take a log resolution $\mu\colon W\to X$ of $(X, \Delta)$ and the linear systems $|D+g^*L+\alpha'|$ and $|p(D+g^*L)|$ where $p$ is a positive integer which is sufficiently divisible such that $\alpha^{\otimes p}$ is trivial and $p$ is divided by $d$ as in the following diagram.
	\begin{center}
	\begin{tikzcd}
		W \arrow[r, "\mu"] & X \arrow[r, "a_X"] \arrow[d, "f"] \arrow[dr, "g"]& \Alb(X) \arrow[d, "a_f"] \\
			 & Y \arrow[r, "a_Y" ] & \Alb(Y) 
	\end{tikzcd}
	\end{center}
Then we have
$$K_W+\Delta_W\sim_\mathbb{Q}\mu^*(K_{X}+\Delta)+E,$$
where the $\Q$-divisors $\Delta_W$ and $E$ are effective and have no common components and $E$ is $\mu$-exceptional. We also have that
$$\mu^*|D+g^*L+\alpha'|=|H_1|+F_1\quad \text{and} \quad \mu^*|p(D+g^*L)|=|H_p|+F_p,$$ where the linear systems $|H_1|$ and $|H_p|$ are base point free, the effective divisors $F_1$ and $F_p$ are the fixed parts of the corresponding linear systems and $\Delta_W+E+F_1+F_p$ has simple normal crossings support. We observe that $pF_1\geq F_p$ and $\frac{p}{d}\mu^*B\geq F_p$. We define a divisor $T$ as
$$T:=\bigg\lfloor \Delta_W-E+\frac{m-1}{mp}F_p\bigg\rfloor.$$
We have that $T\geq\lfloor-E\rfloor$ and thus $-T\leq\lceil E\rceil$. We also have that
$$T\leq\bigg\lfloor \Delta_W+\frac{m-1}{mp}F_p\bigg\rfloor\leq \bigg\lfloor\Delta_W+\frac{m-1}{m}F_1\bigg\rfloor\leq F_1$$
since $(W, \Delta_W)$ is klt. We denote $D+g^*L+\alpha'$ by $D'$. Thus we have
	\begin{center}
	\begin{tikzcd}
		H^0(W, \mathcal{O}_W(\mu^*D'-F_1)) \arrow[r, "a_1"] \arrow[d, "a_2"] &  H^0(W, \mathcal{O}_W(\mu^*D'-T))\arrow[d, "a_3"] \\
			  H^0(W, \mathcal{O}_W(\mu^*D')) \arrow[r, "a_4"] & H^0(W, \mathcal{O}_W(\mu^*D'+\lceil E\rceil)) 
	\end{tikzcd}
	\end{center}
where $a_2$ and $a_4$ are isomorphisms since $F_1$ is the fixed part of the linear system $\mu^*|D'|$ and $\lceil E\rceil$ is $\mu$-exceptional and effective. We also know $a_1$ and $a_3$ are injective and thus $a_1$ and $a_3$ are isomorphisms. Thus we have 
$$H^0(X, \mathcal{O}_X(D'))\cong H^0(W, \mathcal{O}_W(\mu^*D'-T)).$$
Since $\alpha$ is a torsion line bundle, we have $\alpha'\sim_{\Q}0$. For any $\varepsilon\in\Q$, we have 
$$\mu^*(D+\alpha')\sim_{\Q}\frac{1}{m}\mu^*D+(\frac{m-1}{m}-\varepsilon)\mu^*D+\varepsilon\mu^*D$$
$$\sim_{\Q}\mu^*(K_X+\Delta)+\frac{m-1-m\varepsilon}{m}(\frac{1}{p}(H_p+F_p)-\mu^*g^*L)+\varepsilon\mu^*(\frac{1}{d}(g^*H+B)-g^*L)$$
$$\sim_{\Q}\frac{\varepsilon}{d}\mu^*g^*H+K_W+\Delta_W-E+\frac{\varepsilon}{d}\mu^*B-\frac{m-1}{m}\mu^*g^*L+\frac{m-1-m\varepsilon}{mp}(H_p+F_p).$$
Thus we deduce that
$$\mu^*D'-T\sim_{\Q}\mu^*(D+g^*L+\alpha')-T$$
$$\sim_{\Q}\mu^*g^*(\frac{\varepsilon}{d}H+\frac{1}{m}L)+K_W+\Delta_W-E+\frac{\varepsilon}{d}\mu^*B+\frac{m-1-m\varepsilon}{mp}(H_p+F_p)-T.$$
Since $|H_p|$ is base point free and $\Delta_W+E+F_p$ has simple normal crossings support, we can choose a reduced and effective divisor $E_p\sim H_p$ such that $(W, \frac{m-1}{mp}E_p+\Delta_W-E+\frac{m-1}{mp}F_p-T)$ is klt by the definition of $T$. Since $\frac{p}{d}\mu^*B\geq F_p$ and $m>1$, we can choose $\varepsilon\geq0$ sufficiently small and deduce that
$$\Delta'_W(\varepsilon):=\Delta_W-E+\frac{\varepsilon}{d}\mu^*B+\frac{m-1-m\varepsilon}{mp}(E_p+F_p)-T$$
$$\geq\Delta_W-E+\frac{\varepsilon}{p}F_p+\frac{m-1-m\varepsilon}{mp}(E_p+F_p)-\bigg\lfloor \Delta_W-E+\frac{m-1}{mp}F_p\bigg\rfloor$$
$$=\frac{m-1-m\varepsilon}{mp}E_p+\Delta_W-E+\frac{m-1}{mp}F_p-\bigg\lfloor \Delta_W-E+\frac{m-1}{mp}F_p\bigg\rfloor\geq0.$$
Thus we can choose $\varepsilon>0$ sufficiently small such that $(W, \Delta'_W(\varepsilon))$ is klt since $(W, \Delta'_W(0))$ is klt and $\Delta'_W(\varepsilon)\geq0$. We identify $\Pic^0(Y)$ and $\Pic^0(\Alb(Y))$ by $a_Y^*$. We use $\beta'$ to denote the Cartier divisor corresponding to the line bundle $\beta$. We deduce that
$$\mu^*g^*\beta'+\mu^*D'-T\sim_{\Q}\mu^*g^*(\frac{\varepsilon}{d}H+\frac{1}{m}L+\beta')+K_W+\Delta'_W(\varepsilon)$$
where $\frac{\varepsilon}{d}H+\frac{1}{m}L+\beta'$ is an ample $\Q$-divisor on $\Alb(Y)$ since $\varepsilon>0$, $L$ is nef and $\beta'$ is numerically trivial. By \cite[Corollary 10.15]{Kol95}, we have that 
$$H^i(\Alb(Y), \beta\otimes R^j(g\circ\mu)_*\mathcal{O}_W(\mu^*D'-T))=0$$
for every $i>0$, every $j\geq0$ and every $\beta\in\Pic^0(Y)$. Denote the sheaf $(g\circ\mu)_*\mathcal{O}_W(\mu^*D'-T)$ by $\mathcal{F}$. We deduce that 
$$h^0(\Alb(Y), \mathcal{F})=\chi(\Alb(Y), \mathcal{F})=\chi(\Alb(Y), \beta\otimes\mathcal{F})=h^0(\Alb(Y), \beta\otimes\mathcal{F})$$
for every $\beta\in\Pic^0(Y)$.
Since $-T\leq\lceil E\rceil$, we have that
$$h^0(X, \mathcal{O}_X(D'))=h^0(W, \mathcal{O}_W(\mu^*D'-T))=h^0(\Alb(Y), \mathcal{F})$$
$$=h^0(\Alb(Y), \beta\otimes\mathcal{F})=h^0(W, \mu^*g^*\beta\otimes\mathcal{O}_W(\mu^*D'-T))$$
$$\leq h^0(W, \mu^*g^*\beta\otimes\mathcal{O}_W(\mu^*D'+\lceil E\rceil))=h^0(X, g^*\beta\otimes\mathcal{O}_X(D')).$$
Recall that $D'=D+g^*L+\alpha'$. Since the inequality above is true for every $\beta\in\Pic^0(Y)$ and every nef divisor $L$ on $\Alb(Y)$, we can replace $L$ with $L-\beta'$ and get
$$h^0(X, g^*\beta^{-1}\otimes\mathcal{O}_X(D'))\leq h^0(X, \mathcal{O}_X(D'))\leq h^0(X, g^*\beta\otimes\mathcal{O}_X(D'))$$
for every $\beta\in\Pic^0(Y)$. Thus we conclude that for every $\beta\in\Pic^0(Y)$
$$h^0(X, \mathcal{O}_X(D+g^*L)\otimes\alpha)=h^0(X, \mathcal{O}_X(D+g^*L)\otimes\alpha\otimes f^*\beta).$$

We now prove statement (ii). We choose a positive integer $N$ sufficiently divisible such that $h^0(F, \mathcal{O}_F(ND|_{F}))>0$ where $F$ is the general fiber of $f$ such that $(F, \Delta|_{F})$ is a log smooth klt pair. We have $ND|_{F}\sim_{\Q}Nm(K_F+\Delta|_{F})$. The cohomological support locus 
$$V^0(F, \mathcal{O}_F(ND|_{F}))=\{\gamma\in\Pic^0(F)\mid\dim H^0(F, \mathcal{O}_F(ND|_{F})\otimes\gamma)\geq 1\}$$ 
is a finite union of torsion subvarieties of $\Pic^0(F)$ by Corollary \ref{coro1}. We claim that if $\alpha\in V^0(X, \mathcal{O}_X(D))$, then $(\alpha|_{F})^{\otimes N}\cong \mathcal{O}_F$ if $F$ is sufficiently general. Since $\alpha\in V^0(X, \mathcal{O}_X(D))$, we deduce that $(\alpha|_{F})^{\otimes N}\in V^0(F, \mathcal{O}_F(ND|_{F}))$ if $F$ is sufficiently general. If $(\alpha|_{F})^{\otimes N}$ is not trivial, then $V^0(F, \mathcal{O}_F(ND|_{F}))$ contains two different points since it contains $\mathcal{O}_F$ and thus it contains a torsion point $\gamma$ which is not trivial by the structure of $V^0(F, \mathcal{O}_F(ND|_{F}))$. However, $\gamma$ can only be trivial by $\kappa(F, D|_{F})=0$ and it is a contradiction. We conclude that $(\alpha|_{F})^{\otimes N}\cong \mathcal{O}_F$ and by \cite[Lemma 2.6]{CH04} we know that a nonzero multiple of $\alpha$ belongs to $f^*\Pic^0(Y)$. Thus there exist countably many torsion points $\alpha_j\in\Pic^0(X)$ such that
$$V^0(X, \mathcal{O}_X(D))\subseteq\bigcup_{j\in J}(\alpha_j\otimes f^*\Pic^0(Y))$$
where $J$ is a countable index set. Since $V^0(X, \mathcal{O}_X(D))$ is a closed subset of $\Pic^0(X)$, it can be decomposed as a union of irreducible components
$$V^0(X, \mathcal{O}_X(D))=\bigcup_{k\in K}Z_k$$
where $K$ is a finite index set. We deduce that
$$Z_k=\bigcup_{j\in J}(Z_k\cap(\alpha_j\otimes f^*\Pic^0(Y))).$$
Note that we are only considering closed points of varieties now. Since an irreducible variety over an uncountable and algebraically closed field cannot be a countable union of proper closed subsets, we deduce that $Z_k\subseteq\alpha_{j_k}\otimes f^*\Pic^0(Y)$ for some $j_k\in J$. By statement (i) of our theorem, we have 
$$Z_k\subseteq\alpha_{j_k}\otimes f^*\Pic^0(Y)\subseteq V^0(X, \mathcal{O}_X(D))$$
and thus $Z_k=\alpha_{j_k}\otimes f^*\Pic^0(Y)$ since $Z_k$ is an irreducible component of $V^0(X, \mathcal{O}_X(D))$. Then we deduce that 
$$V^0(X, \mathcal{O}_X(D))=\bigcup_{k\in K}(\alpha_{j_k}\otimes f^*\Pic^0(Y)).$$
\end{proof}

We are ready to prove Theorem \ref{main4} now.

\begin{proof}[Proof of Theorem \ref{main4}]
By the same argument of Theorem \ref{main6}, we can assume $(X, \Delta)$ is log smooth. By Theorem \ref{main3}, we know $(a_X)_*\mathcal{O}_X(lD)$ admits the Chen-Jiang decomposition for every positive integer $l$
$$(a_X)_*\mathcal{O}_X(lD)\cong \bigoplus_{i\in I}(\alpha_i\otimes p_i^*\mathcal{G}_i)$$
where each $A_i$ is an abelian variety, each $p_i\colon \Alb(X)\to A_i$ is a fibration, each $\mathcal{G}_i$ is a nonzero M-regular coherent sheaf on $A_i$, and each $\alpha_i\in\Pic^0(X)$ is a torsion line bundle whose order is bounded independently of $l$. We identify $\Pic^0(X)$ and $\Pic^0(\Alb(X))$ by $a_X^*$. By \cite[Lemma 3.3]{LPS20}, we have that
$$V^0(X, \mathcal{O}_X(lD))=\bigcup_{i\in I}(\alpha^{-1}_i\otimes p_i^*\Pic^0(A_i)).$$
We know that $V^0(X, \mathcal{O}_X(lD))$ is a finite union of torsion translates of $(a_f)^*\Pic^0(Y)$ by Theorem \ref{main6} and thus for every $i\in I$, we have a factorization
	\begin{center}
	\begin{tikzcd}
		\Alb(X) \arrow[r, "a_f"] \arrow[rr, bend right, "p_i"]& \Alb(Y) \arrow[r, "q_i"] & A_i
	\end{tikzcd}
	\end{center}
where each $q_i$ is a fibration since each $p_i$ is a fibration. We define $\mathcal{F}_i=q_i^*\mathcal{G}_i$ and then we have
$$(a_X)_*\mathcal{O}_X(lD)\cong \bigoplus_{i\in I}(\alpha_i\otimes a_f^*\mathcal{F}_i).$$
We claim that each $\mathcal{F}_i$ satisfies $\IT$. Fix a torsion point $\alpha\in\Pic^0(X)$. By Theorem \ref{main6}, $h^0(X, \mathcal{O}_X(lD)\otimes\alpha\otimes f^*\beta)$ is constant for every $\beta\in\Pic^0(Y)$. We know
$$h^0(X, \mathcal{O}_X(lD)\otimes\alpha\otimes f^*\beta)=\sum_{i\in I}h^0(\Alb(X), \alpha\otimes\alpha_i\otimes a_f^*\mathcal{F}_i\otimes a_f^*\beta)$$
and thus each term on the right is constant for every $\beta\in\Pic^0(Y)$ since those terms are upper semi-continuous functions with respect to $\beta$. In particular, we can take $\alpha$ to be $\alpha_i^{-1}$ and deduce that $h^0(\Alb(X), a_f^*\mathcal{F}_i\otimes a_f^*\beta)$ is constant for every $\beta\in\Pic^0(Y)$. By the base change theorem, we know that
$$h^0(\Alb(X), a_f^*\mathcal{F}_i\otimes a_f^*\beta)=h^0(\Alb(Y), \mathcal{F}_i\otimes\beta)$$
since $a_f$ is a fibration by \cite[Lemma 2.6]{CH04}. Thus $h^0(\Alb(Y), \mathcal{F}_i\otimes\beta)$ is constant for every $\beta\in\Pic^0(Y)$. Next we prove that each $\mathcal{F}_i$ is a GV-sheaf. We can take an isogeny $k\colon B\to\Alb(X)$ such that each $k^*\alpha_i$ is trivial and consider the following commutative diagram for the fiber product $X':=X\times_{\Alb(X)}B$.
	\begin{center}
	\begin{tikzcd}
			X' \arrow[r, "p"] \arrow[d, "q"] & B \arrow[d, "k"] \arrow[dr, "k'"]\\
			 X \arrow[r, "a_X" ] & \Alb(X) \arrow[r, "a_f"] &\Alb(Y)
	\end{tikzcd}
	\end{center}
The morphisms $p$ and $q$ are the projections and $q$ is \'{e}tale. Consider the log smooth klt pair $(X', \Delta')$ defined by $K_{X'}+\Delta'=q^*(K_{X}+\Delta)$. By Corollary \ref{coro1}, we deduce that each $\mathcal{F}_i$ is a GV-sheaf. Then each $\mathcal{F}_i$ satisfies $\IT$ by Proposition \ref{IT}.
\end{proof}

The next corollary is a special case of Theorem \ref{main4}, which applies to klt pairs of log general type.

\begin{coro}\label{ITL}
Let $f\colon X\to Y$ be a smooth model of the Iitaka fibration associated to a Cartier divisor $D$ on $X$ where $D\sim_{\Q}m(K_X+\Delta)$, $m>1$ is a rational number, $Y$ is smooth and $(X, \Delta)$ is a klt pair. Let $a_X\colon X\to \Alb(X)$ be the Albanese morphism of $X$. If $q(X)=q(Y)$ and $(a_X)_*\mathcal{O}_X(D)\neq0$, then:
\begin{enumerate}
	\item[$\mathrm{(i)}$] $V^0(X, \mathcal{O}_X(D))=\Pic^0(X)$.

	\item[$\mathrm{(ii)}$] $(a_X)_*\mathcal{O}_X(D)$ satisfies $\IT$, hence in particular it is ample. 
\end{enumerate}
\end{coro}

\begin{proof}
By the same argument of Theorem \ref{main6}, we can assume $(X, \Delta)$ is log smooth. Since $(a_X)_*\mathcal{O}_X(D)\neq0$ is a GV-sheaf by Corollary \ref{coro1}, we deduce $V^0(X, \mathcal{O}_X(D))$ is not empty by \cite[Lemma 7.4]{HPS18}. Since $q(X)=q(Y)$, the fibration $a_f$ is an isomorphism. Then statement (i) follows from Theorem \ref{main6}. By Theorem \ref{main4}, $(a_X)_*\mathcal{O}_X(D)$ satisfies $\IT$. It is ample by \cite[Proposition 2.13]{PP03} and \cite[Corollary 3.2]{Deb06}.
\end{proof}

Next we use Corollary \ref{ITL} to give effective bounds for the generation of klt pairs of log general type on irregular varieties. This extends some results in \cite{PP03, PP11b, LPS20} to klt pairs. Theorem \ref{main5} follows from Corollary \ref{ITL} and standard arguments. 

\begin{thm}\label{main5}
Let $f\colon X\to Y$ be a smooth model of the Iitaka fibration associated to a Cartier divisor $D$ on $X$ where $D\sim_{\Q}m(K_X+\Delta)$, $m>1$ is a rational number, $Y$ is smooth and $(X, \Delta)$ is a klt pair. Let $a_X\colon X\to \Alb(X)$ be the Albanese morphism of $X$. Assume that $q(X)=q(Y)$ and there exists a nonempty open subset $W\subseteq\Alb(X)$ consisting of points $a$ such that $\mathcal{O}_{X_a}(D)$ is globally generated and the natural map
$$(a_X)_*\mathcal{O}_{X}(D)\otimes\C(a)\to H^0(X_a, \mathcal{O}_{X_a}(D))$$
is surjective. Then:
\begin{enumerate}
	\item[$\mathrm{(i)}$] $\mathcal{O}_{X}(2D)\otimes\alpha$ is globally generated on $a_X^{-1}(W)$ for every $\alpha\in\Pic^0(X)$.

	\item[$\mathrm{(ii)}$] Assuming in addition that for every $x\in a_X^{-1}(W)$ and every $a\in W$, $\mathcal{O}_{X_a}(2D)\otimes\mathcal{I}_{x}|_{X_a}$ is globally generated and the natural map
$$\quad (a_X)_*(\mathcal{O}_{X}(2D)\otimes\mathcal{I}_{x})\otimes\C(a)\to H^0(X_a, \mathcal{O}_{X_a}(2D)\otimes\mathcal{I}_{x}|_{X_a})$$
is surjective, then $\mathcal{O}_{X}(3D)\otimes\alpha$ is very ample on $a_X^{-1}(W)$ for every $\alpha\in\Pic^0(X)$.
\end{enumerate}
\end{thm}

\begin{proof}
It follows from the same proof of \cite[Theorem 12.2]{LPS20} and we include some details for the reader's convenience. We can assume that $a_X^{-1}(W)$ is not empty. We can identify $\Pic^0(X)$ and $\Pic^0(\Alb(X))$ using $a_X^*$ by \cite[Lemma 8.1]{Kaw85a}. We prove statement (i) first. We know $(a_X)_*\mathcal{O}_{X}(D)\neq0$ by assumptions and thus it satisfies $\IT$ by Corollary \ref{ITL}. Take a $\beta\in\Pic^0(X)$ such that $\beta^{\otimes2}\cong\alpha$. Then $(a_X)_*(\mathcal{O}_{X}(D)\otimes\beta)$ satisfies $\IT$. By \cite[Proposition 2.13]{PP03}, it is continuously globally generated which means that there exists an integer $N$ such that for general line bundles $\alpha_1, \dots, \alpha_N\in\Pic^0(\Alb(X))$ the sum of the twisted evaluation morphisms
$$\bigoplus_{i=1}^NH^0(\Alb(X), (a_X)_*(\mathcal{O}_{X}(D)\otimes\beta)\otimes\alpha_i)\otimes\alpha_i^{-1}\to (a_X)_*(\mathcal{O}_{X}(D)\otimes\beta)$$
is surjective. We know $a_X^*(a_X)_*\mathcal{O}_{X}(D)\to\mathcal{O}_{X}(D)$ is surjective on $a_X^{-1}(W)$ by assumptions. Thus 
$$\bigoplus_{i=1}^NH^0(X, \mathcal{O}_{X}(D)\otimes\beta\otimes a_X^*\alpha_i)\otimes a_X^*\alpha_i^{-1}\to \mathcal{O}_{X}(D)\otimes\beta$$
is surjective on $a_X^{-1}(W)$. We conclude that 
$$\mathcal{O}_{X}(2D)\otimes\alpha\cong(\mathcal{O}_{X}(D)\otimes\beta)^{\otimes2}$$
is generated by global sections on $a_X^{-1}(W)$ by the same argument of \cite[Proposition 2.12]{PP03}.

We prove statement (ii) now. By a similar argument as above and the additional assumptions, we only need to prove $(a_X)_*(\mathcal{O}_{X}(2D)\otimes\mathcal{I}_{x})$ satisfies $\IT$ for every $x\in a_X^{-1}(W)$. By Corollary \ref{ITL}, $(a_X)_*(\mathcal{O}_{X}(2D))$ satisfies $\IT$. Since $\mathcal{O}_{X}(2D)$ is generated by global sections on $a_X^{-1}(W)$ by statement (i) of our theorem, we can prove that $(a_X)_*(\mathcal{O}_{X}(2D)\otimes\mathcal{I}_{x})$ satisfies $\IT$ for every $x\in a_X^{-1}(W)$ by a standard argument using exact sequences.
\end{proof}

Corollary \ref{coro5} follows from Theorem \ref{main5} directly.

\begin{proof}[Proof of Corollary \ref{coro5}]
It follows from Corollary \ref{ITL} and Theorem \ref{main5} since the klt pair $(X, \Delta)$ is of log general type and $a_X$ is generically finite onto its image.
\end{proof}

	\bibliographystyle{amsalpha}
	\bibliography{biblio}	

\end{document}